\documentclass[a4paper, 11pt]{article}
\usepackage{natbib, a4wide}

\usepackage{amsfonts}
\usepackage{amssymb}
\usepackage{graphicx}
\usepackage{amsmath}

\newcommand{\cH}{{\mathcal H}}
\newcommand{\cS}{{\mathcal S}} 
\newcommand{\cP}{{\mathcal P}}
\newcommand{\cL}{{\mathcal L}}

\newcommand{\bN}{{\mathbb N}}
\newcommand{\bS}{{\mathbb S}}

\newcommand\eref[1]{(\ref{#1})}
\newcommand{\bT}{{\mathbb T}}

\def\bR{{\mathbb{R}}} 

\newtheorem{theo}{Theorem}
\newtheorem{prop}{Proposition} 

\newtheorem{lemma}{Lemma}

\newtheorem{condition}{Condition}

\newtheorem{remark}{Remark}
\newenvironment{proof}[1][Proof]{\noindent\textbf{#1.} }{\ \rule{0.5em}{0.5em}}

\begin{document}

\title{Concentration Inequalities and Confidence Bands for Needlet Density Estimators on Compact Homogeneous Manifolds}

\author{\textsc{Gerard Kerkyacharian, Richard Nickl} and \textsc{Dominique Picard} \\ \\ \textit{LPMA, University of Cambridge and Universit\'{e} Paris Diderot}}

\date{This Version: February 2011, First Version: November 2010}

\maketitle

\begin{abstract}
Let $X_1,...,X_n$ be a random sample from some unknown probability density $f$ defined on a compact homogeneous manifold $\mathbf M$ of dimension $d \ge 1$. Consider a 'needlet frame' $\{\phi_{j \eta}\}$ describing a localised projection onto the space of eigenfunctions of the Laplace operator on $\mathbf M$ with corresponding eigenvalues less than $2^{2j}$, as constructed in \cite{GP10}. We prove non-asymptotic concentration inequalities for the uniform deviations of the linear needlet density estimator $f_n(j)$ obtained from an empirical estimate of the needlet projection $\sum_\eta \phi_{j \eta } \int f \phi_{j \eta }$ of $f$. We apply these results to construct risk-adaptive estimators and nonasymptotic confidence bands for the unknown density $f$. The confidence bands are adaptive over classes of differentiable and H\"{older}-continuous functions on $\mathbf M$ that attain their H\"{o}lder exponents. 

\textit{MSC 2000:} 62G07, 60E15, 42C40
\end{abstract}


\section{Introduction}

We consider the problem of constructing confidence bands for an unknown probability density $f$ based on a sample $X_1, . . . ,X_n$ from $f$ observed on the $d$-dimensional compact homogeneous manifold $\mathbf M$.  The classical statistical applications occur when $\mathbf M$ equals the $d$-dimensional unit sphere $\mathbb S^d$ of $\mathbb R^{d+1}$: If $d=1$ this corresponds to estimating a periodic univariate density, and recent interest lies mostly in the case $d=2$, strongly motivated by statistical problems in astrophysics, see \cite{density} for an account of typical problems and applications in astrophysics and directional statistics more generally. In \cite{density} a recent construction of wavelet type bases on $\mathbb S^d$ -- due to \cite{NPW, pnarco}, who called these new basis functions \textit{needlets} -- was employed to construct risk-adaptive estimators for $f(x), x \in \mathbb S^d,$ by a local needlet series with support concentrated in a neighborhood of $x$. See also \cite{spheredeconv} for similar results in the spherical deconvolution problem. The main advantages of this approach are that they share none of the drawbacks of classical approaches: kernel methods do not take the manifold structure of the sphere well into account, orthogonal series methods associated with spherical harmonics have very poor pointwise (and even worse uniform) performance since spherical harmonics are not well localized but spread out all over the sphere, and methods based on stereographic projections of the sphere onto the plane use a distorted approximation-theoretic paradigm. In contrast needlets are a tight frame constructed on the spherical harmonics which are highly localized and allow for optimal approximation not only in $L^2$ but in general $L^p$-spaces, including in particular $L^\infty$, which is particularly relevant in the problem of constructing confidence bands for $f$. Moreover the localization property is of crucial importance since in astrophysical data sets some parts of the sphere (sky) may not be covered by the observations, so that non-local procedures may suffer severely from missing data points. 

The main contributions of the present article are three-fold. First, building on recent results on wavelets and approximation of functions on manifolds in \cite{GM09, GP10}, we show how needlet estimators $f_n(j,y), y \in \mathbf M,$  with resolution level $j \ge 0,$ can be defined also on the more general class of compact homogeneous differentiable manifolds $\mathbf M$, which includes, next to $d$-dimensional unit spheres, also other relevant examples such as real and complex projective spaces, or Grassmann and Stiefel manifolds. The main idea behind this construction is to use tools from harmonic analysis on compact Lie groups that allow to build a localized frame on the eigenfunctions of a second order elliptic differential Laplace operator on $\mathbf M$, which in the case of the sphere coincides with the construction of \cite{NPW, pnarco}, where these eigenfunctions are precisely the spherical harmonics.

The second goal of this article is to prove non-asymptotic concentration inequalities for the uniform fluctuations $$\sup_{y \in \mathbf M}|f_n(j,y)-Ef_n(j,y)|$$ of needlet estimators $f_n(j)$ around the needlet projections $Ef_n(j)=A_j(f)$ of the unknown density $f$. The constants in these concentration inequalities depend in a natural way on the manifold and we derive reasonably tight constants for the case $\mathbf M = \mathbb S^d, d \ge 1$. We present both Bernstein-type bounds and inequalities based on Rademacher-symmetrization in a similar vein as in recent work in \cite{KolaAOS, GineNickl10b, ln10}.

The third goal is to use the above concentration inequalities to construct estimators and confidence bands for the unknown density $f:\mathbf M \to \mathbb R$. Even the problem of spherical confidence bands seems not to have been addressed in the literature so far -- one reason may arise from the fact that the classical approach in the univariate case (\cite{bros}) via extreme value theory does not straightforwardly generalise to sample spaces with a different geometric structure. Our concentration inequalities hold on arbitrary compact homogeneous manifolds and can be used directly to construct estimators and nonasymptotic confidence bands for the unknown density $f$ if one has apriori control of the approximation error of $f$ by its needlet projection $A_j(f)$ (the 'bias' of estimation), which by results in \cite{GP10} is equivalent to classical H\"{o}lderian smoothness conditions for $f$ on $\mathbf M$. 

Since knowledge of the bias is usually not available, the question of how to choose $j$ comes into sight, and to which extent \textit{adaptive} estimators and confidence bands can be constructed. It is known on the one hand (\cite{Lowconf}) that adaptive and honest confidence bands in nonparametric function estimation problems cannot exist over the entirety of the usual smoothness classes (in our case, H\"{o}lder-balls on $\mathbf M$). Recent work in this field, however, can be interpreted as a new way of looking at this problem: One can devise statistically relevant subsets of the usual smoothness function classes for which adaptive confidence bands \textit{do exist}. One example comes from shape constrained nonparametric regression, see, e.g., \cite{dumbgen}. Other examples are 'self-similar functions' that attain their H\"{o}lder exponent -- see \cite{Pictriboul} in the case of the Gaussian white noise model and regression framework and \cite{GineNickl10a} in density estimation on the real line. Moreover, building on \cite{jaffard}'s work on the Frisch-Parisi conjecture (\cite{FrisParisi}), \cite{GineNickl10a} proved that 'generic' subsets (in the Baire-sense) of the class of H\"{o}lder balls can be constructed for which asymptotically honest adaptive confidence bands exist. 

In the present paper we follow the line of  \cite{Pictriboul} and \cite{GineNickl10a}, but take a nonasymptotic approach. We propose an adaptive procedure $\hat j_n$ based on Lepski's method (\cite{Lepski})  to choose the resolution level $j$ for the needlet estimator $f_n(j)$ in a data-driven way. The resulting estimator $f_n(\hat j_n)$ adapts to the unknown smoothness of $f$ in sup-norm risk. In our main result we devise an analytic condition on the approximation errors of $f$ by its needlet projections $A_j(f)$ under which we can establish both an asymptotic and a nonasymptotic coverage result for confidence bands for $f$ over arbitrary subsets $\Omega$ of $\mathbf M$ that are centered at $f_n(\hat j_n)$, and we show that this band adapts to the unknown smoothness of $f$ in the minimax sense. Intuitively the results in \cite{GineNickl10a} suggest that adaptation is possible for functions $f: \mathbf M \to \bR$ that attain their H\"{o}lder exponent, and indeed we prove that our analytic condition can be interpreted in terms of classical H\"{o}lder regularity properties of $f$. The proof of this result is somewhat delicate and we detail it only in the case $\bS^d$, where the representation of the projector onto spherical harmonics in terms of Gegenbauer polynomials allows for explicit derivations. 

Let us finally remark that even in the univariate case $\mathbb S^1$ our nonasymptotic approach to confidence bands gives an alternative to the more classical asymptotic techniques based on extreme value theory, as initiated in the classical paper \cite{bros}, and as also used in the adaptive context in \cite{GineNickl10a}. Not surprisingly the results obtained via a nonasymptotic approach have limitations, but in contrast to the classical asymptotic theory referred to above, the present results give precise conditions for what is necessary to obtain coverage in finite samples.

\section{Compact Homogeneous Manifolds and Needlets} \label{need}

We summarize here some facts on compact homogeneous manifolds and Lie groups (see \cite{H78, H00, W83, Faraut} for general references), and the construction and essential properties of the associated needlet frame due to \cite{GM09, GP10}, generalising the spherical case considered in \cite{NPW}.

\subsection{Compact Lie Groups and the Laplace Operator}

Let $\bf M$ be a compact connected differentiable ($C^\infty$-) manifold of dimension $dim(\mathbf M)=d$. A compact Lie group $G$ of dimension $\tau$ is said to act on $\bf M$ via $$(g,x) \in G \times \mathbf {M} \mapsto g.x \in \bf M$$ if a) this action is, for every $g \in G$, a diffeomorphism of $\bf M$,  if b) $g_1g_2 . x = g_1 . (g_2 . x)$ holds for every $g_1, g_2 \in G, x \in \bf M$, if c) the identity $e \in G$ satisfies $e . x =x$ and if d) for every $g \in G, g \ne e,$ there exists a point $x \in \bf M$ such that $g. x \ne x$. A group $G$ acts \textit{transitively} on $\bf M$ if in addition $$\text{ for every } x,y \in \mathbf M ~ \text{ there exists } g \in G \text{ s.t. } g . x =y.$$ A compact manifold $\bf M$ is said to be \textit{homogeneous} if it is a compact connected differentiable manifold on which a compact Lie group acts transitively. Examples include the $d$-dimensional unit sphere $\mathbb S^d$ of $\mathbb R^{d+1}$, projective spaces, Stiefel and Grassmann manifolds, see p.125 in \cite{W83} and also \cite{W52} for the two-point homogeneous case.

Any compact homogeneous manifold $\mathbf M$ can be realised as a quotient $G/K$ where $K$ is a closed subgroup of $G$. More precisely, if we fix once and for all a point $x_0 \in \mathbf M,$ and let $K= \{h \in G, h.x_0 = x_0\}$ be the closed isotropy subgroup at $x_0$, then $\mathbf M$ is diffeomorphic to $G/K$
and the canonical projection $\pi : g \in G \mapsto \overline g=\{gh, h \in K\} \in G/K$ is  continuous, onto and verifies $ \pi (g_1 g_2)  =  g_1 \pi ( g_2) $, see \cite{W83}, p.123 onwards. Moreover the image of the Haar measure on $G$ under $\pi$,
$$ \int_{G} f( \pi(g)) d{g} = \int_{G/K} f(x) dx  = \int_\mathbf M f(x) dx,$$
is a natural "Haar" measure $dx$ on $\mathbf M$, invariant under the action of $G$. (It is the unique $G-$invariant measure on $\bf M$ up to a scaling factor.) The usual Lebesgue spaces on $\bf M$ are denoted by $L^p(\mathbf M):=L^p(\mathbf M, dx), 1 \le p \le \infty$. Since $G$ is compact, $dx$ is bi-invariant: for $f\in L^1(\mathbf M)$ and $g \in G$ let us define $L_g(f)(x) = f(g^{-1}x), R_g(f)(x)=f(xg),$ then $$ \int_\mathbf M L_g(f)(x) dx = \int_\mathbf M f(x) dx = \int_\mathbf M R_g(f)(x)dx.$$
The Lie algebra $Lie(G)$ of $G$ is characterized by the fact that
$$ X \in  Lie( G ) \mapsto e^{X} \in G,$$ and since $G$ is compact, this mapping is onto. Let us recall that we have the $Ad$ representation of $G$ in $Lie(G):$$$  g \in G \mapsto  Ad(g) X \equiv gXg^{-1}  \in Lie(G),  \; \hbox{and}\;    g e^X g^{-1} = e^{Ad(g) X},$$
and there exists an Euclidean structure $\langle \cdot, \cdot \rangle$ on $Lie(G)$ for which $Ad$ is unitary, that is, such that 
\begin{equation}\label{AD}
 \forall g \in G, \quad \forall X \in Lie(G), \quad \langle Ad(g) X, Ad(g) Y \rangle = \langle  X,  Y \rangle, \quad |X|^2 = \langle  X,  X \rangle,
 \end{equation}
see Proposition 6.1.1 in \cite{Faraut}. 

Every $X \in Lie(G)$ generates a vector field on $G$ so that we can define a one parameter group $$t \mapsto e^{tX} \in G, t \in \bR,$$
and since $G$ is connected we can define a metric on $G$ by the 'length' $|X|$ of the 'shortest geodesic' joining two points $g_1, g_2 \in G$, 
\begin{equation}\label{dG}
 d_G(g_1,g_2) = \inf \{|X|, \; e^X g_1 = g_2\} = \inf \{|X|, \;  g_1 e^X = g_2\}. 
\end{equation}
The two previous definitions are equivalent, as :
$$ e^X g_1 = g_2 \Longleftrightarrow   g_1 g_1^{-1} e^X g_1 = g_2\Longleftrightarrow   g_1   e^{Ad(g_1^{-1})X}  = g_2, \;   |Ad(g_1^{-1})X| = |X|$$
and it is not difficult to verify that this metric is bi-invariant :
$$ \forall g_1, g_2 , g \in G, \;
 d_G(g_1,g_2) =  d_G( gg_1, gg_2)  = d_G(g_1g ,g_2g). $$
Every $X \in Lie(G)$ also naturally generates a one parameter group on $\mathbf M:$
 $$ t  \in \bR \mapsto e^{tX}.x \in \mathbf M$$
which describes geodesics of the Riemannian structure on $\mathbf M$ associated to the Euclidean structure $\langle \cdot, \cdot \rangle$ on $Lie(G).$ The metric on $\mathbf M$ is given by
 $$ d_\mathbf M(x,y) = \inf\{ |X|, \; e^X.x =y\} = d_{G/H}(x,y)= \inf\{ d_G (g_1,g_2) , \; \pi(g_1) =x, \pi(g_2) =y \}$$
So $d_\mathbf M(\pi(g), \pi(g')) \leq d_G(g,g').$ Moreover 
$$\forall g \in G, \; x, y \in \mathbf M , \; d_\mathbf M( g.x, g.y) =  d_\mathbf M(x,y).$$ 
This is again due to (\ref{AD}) as 
$$   e^X . x =y \Longleftrightarrow    g. e^X g^{-1}. g . x = g. y \Longleftrightarrow e^{Ad(g) X} g.x = g.y, \; \hbox{ and } \; |X| =|Ad(g) X|.$$
Now similarly every $X \in Lie(G)$ gives rise to a one-parameter group on $L^p(\bf M)$, $1\leq p <\infty$, given by 
$$f \mapsto  T_t(f) (x) =f(e^{tX}.x);~ t \in \mathbb R, x \in \mathbf M,  f  \in L^p(\bf M)$$
and we denote the infinitesimal generator of this one-parameter group by $D_X$, so $$D_Xf(x) = \frac{d}{dt}f(e^{tX} . x)|_{t=0}, ~~x \in \mathbf M,$$ the derivative of $f$ at $x$ in the direction of the $X$-geodesic.

If $X_i, i=1, \dots, \tau$, is an orthonormal basis of $Lie(G)$ with respect to the scalar product induced by the adjoint representation, the sum $$\cL = \sum_{i=1}^\tau X_i^2$$ defines the Casimir operator, which is independent of the choice of the basis, and which is a central element of the enveloping algebra of $Lie(G).$ Associated to the Casimir operator is the following operator on $L^2(\mathbf M)$  (we keep the same notation $\cL$)
$$\mathcal L = D^2_{X_1} + D_{X_2}^2 + \dots + D_{X_\tau}^2.$$
The operator $-\mathcal L$, which is often called the Laplace operator, is a second order, positive,  elliptic differential operator defined on the space $C^\infty(\bf M)$  of infinitely differentiable functions on $\bf M$. In fact $-\mathcal L$ can be closed to give a positive,  self-adjoint second order elliptic differential operator on $L^2(\bf M)$ with a discrete spectrum of eigenvalues $\lambda_k, k \in \mathbb N,$ arranged in increasing and divergent order. By the spectral theorem the corresponding eigenfunctions $\{e_k\}_{k \in \mathbb N}$ constitute an orthonormal basis of $L^2(\bf M)$, and we define, for $n \in \mathbb N$, the closed finite-dimensional subspaces $E_n=E_n(\bf M)$ of $L^2(\bf M)$ spanned by eigenfunctions $e_k$ of $\mathcal L$ whose corresponding eigenvalues $\lambda_k$ do not exceed $n$, formally $$E_n(\mathbf M) := \left\{x \mapsto \sum_{k: \lambda_k \le n} c_k e_k(x): ~c_k \in \mathbb R,~ \lambda_k \text{ an eigenvalue of } e_k  \right\}.$$

\subsection{Connection to the Laplace-Beltrami Operator}

The operator $\mathcal L$ need not necessarily coincide with the Laplace-Beltrami operator on $\mathbf M$, but it does in several important cases. If $M$ is a two-point homogeneous space then $\cL$ equals, up to a scaling constant, the Laplace-Beltrami operator, see Proposition 4.11 in Chapter II of \cite{H00}. By \cite{W52}'s classification of such spaces this includes, among others, the $d$-dimensional unit sphere, real and certain complex projective spaces. Further examples for manifolds where the Laplace-Beltrami operator coincides with $- \mathcal L$ are given in \cite{GP10}. Since this connection is of some interest in applications, we discuss this point here in some more detail.

The Laplace operator $\cL$ is left- and right invariant and symmetric with respect to the inner product $\langle \cdot, \cdot \rangle$ induced by the adjoint representation, see p.162 in \cite{Faraut}. By the general theory of irreducible unitary representation of compact Lie groups (e.g., Theorem 6.4.1 and Proposition 8.2.1 in \cite{Faraut}) :
$$L^2(\mathbf M) = \bigoplus_j V_j, \; V_j = ker (\cL -c_j I)$$
for constants $c_j$, and $\forall g \in G, \; L_g(V_j) \subset V_j,$
$$ g \in G \mapsto L_g  \in Lin(V_j)$$ is a finite dimensional unitary irreducible representation of $G,$ where $Lin (V_j)$ denotes the space of bounded linear operators on $V_j$. 

Moreover, as a Riemannian manifold, $\mathbf M $ is equipped with a Laplace-Beltrami operator $\Delta$ which commutes with the $G-$ action:  $ \forall g \in G, \;  \Delta L_g= L_g \Delta.$ If $M$ is compact :
$$L^2(\mathbf M) = \bigoplus_k \cH_k, \quad \cH_k =ker (\Delta-\lambda_k I).$$
Moreover $\cH_k$ is $ G-$ invariant ($ \forall g \in G, \; L_g( \cH_k) \subset \cH_k $),  so 
$$ g \in G \mapsto L_g  \in Lin(\cH_k)$$
is a finite dimensional unitary  representation of $G.$ 

Clearly, if $\Phi_k(x,y)$ is the kernel of the projection operator onto $\cH_k,$ then $\phi_k(y) = \Phi_k(x_0,y)$  verifies $ \| \phi_k \|_2^2 = \phi_k(x_0) =dim(\cH_k)$ and is moreover a \textit{zonal} function
 (recall that $f$ is zonal if $ \forall h\in K, \; L_h(f) =f$, see, e.g., \cite{G75, H00}).  If the space of zonal funtions in $\cH_k$ is of dimension $1$ then 
 $ g \in G \mapsto L_g  \in Lin(\cH_k)$ is an irreducible representation. If this is the case for all $\cH_k$ then $\cL$ and the Laplace-Beltrami will coincide, if we can check that
 the eigenvalues are the same.
 

Let us illustrate this in the case of $\mathbf M = \mathbb S^d$, where $$G= SO(d+1) = \{A \in M(d+1\times d+1), \; A^{-1}= A^t\},$$ $$ \quad  Lie(G) =so(d+1)=  \{X \in M(d+1\times d+1), \; -X= X^t\}$$
and we can take $$ \langle X,Y\rangle = \frac 12 Tr (XY^t).$$ An orthonormal basis is then given by $$ X_{i,j} =E_{i,j}- E_{j,i}, \;  1\leq i<j \leq d+1, \quad E_{j,i}=(\alpha^{i,j}_{k,l})_{k,l}, \; \alpha^{i,j}_{k,l}=\delta_{i,k} \delta_{j,l} .$$
We take $x_0= (1,0, \dots, 0)$ so $K \approx SO(d)$ and 
$$ \forall x,y \in  M= \mathbb S^d, \; d_{ \mathbb S^d} (x,y) = \arccos (\langle x, y \rangle_{\mathbb R^{d=1}})$$
The eigenvalues  of $\Delta $ are $ \lambda_k = -k(k+d-1)$ (\cite{Faraut}), the space $\cH_k$ equals the space of spherical harmonic functions of degree $k$, and there is only one zonal function in each $\cH_k$ (which is given through Gegenbauer polynomials) so the induced representation are irreducible (and not equivalent). To see that $\Delta =-\cL$ it is enough to compute the eigenvalue of $\cL$ on $\cH_k$ and this can be carried on in the case of the sphere using the explicit expression of $\cL =-\sum_{i<j} D_{X_{i,j}}^2.$

\subsection{A Smoothed Projection onto the Span of the Eigenfunctions of $-\mathcal L$} \label{sh}

We shall write $\langle g, h \rangle$ from now on for the standard $L^2(\bf M)$-inner product of two functions $g,h \in L^2(\mathbf M):=L^2(\mathbf M, dx)$. We also denote by $\|g\|_\Omega=\sup_{y \in \Omega}|g(y)|$ the supremum norm of $g: \bf M \to \mathbb R$ over $\Omega \subseteq \bf M$, and we shall write $\|g\|_\infty$ when $\Omega = \bf M$.

Let $0 \leq a \leq 1$ be an infinitely differentiable  nonnegative function defined on $[0, \infty)$. We require $a$ to be identically $1$ on $[0, 1/2]$ and compactly supported on $[0,1].$ Define the sequence of linear operators $A_j,\; j\ge 0,$ with $$A_0f = \int_\mathbf M f(x) dx, \quad A_jf(x) := A_j (f) (x) =  \int_\mathbf M A_j(x,y) f(y) dy, ~~j > 0,$$ where, for $L_k(x,y)=e_k(x) \overline{e_k(y)}$,
 $$A_j(x,y):=  \sum_{k} a\left(\frac{\lambda_k}{2^{2j}}\right)  L_k(x,y) = \sum_{k: \lambda_k <2^{2j}} a\left(\frac{\lambda_k}{2^{2j}}\right)  e_k(x) \overline {e_k(y)}.$$
Clearly
\begin{equation*}
\langle A_jf,f \rangle = \sum_{k} a\left(\frac {\lambda_k}{2^{2j}}\right) \langle  L_k f, f \rangle \leq   \| f\|_2^2, \quad \|A_jf\|_2 \le \|f\|_2
\end{equation*}
from Parseval's identity and since $|a|\le 1$. Since $a$ is identically one on $[0,1/2]$
\begin{equation} \label{polyid}
h \in E_{2^{2j-1}}(\mathbf M) \text{ implies }A_j(h)=h
\end{equation}
and since $E_n(\mathbf M), n \ge 1,$ is dense in $L^2(\bf M)$ we conclude
\begin{equation*}
{\label{pro2}}
\lim_{j\rightarrow \infty } \| A_jf-f\|_2=0 
\end{equation*}
for every $f \in L^2(\bf M)$. Thus $A_j$ furnishes us with an approximation of the identity operator on $L^2(\bf M)$. 

The kernel $A$ can be 'split' as follows: If we define
\begin{align*}C_{j} (x,y)&= \sum_{k: \lambda_k <2^{2j}}   \sqrt {a \left(\frac{\lambda_k}{2^{2j}}\right)} L_k(x,y)\; \end{align*}
then due to the orthogonality properties of the $L_k$'s we see
\begin{equation} \label{aj}
 A_{j}(x,y)=\int_\mathbf M C_{j} (x,u) C_{j} (u,y) du.
\end{equation}

\subsection{Gauss Cubature Formula and Needlets on a Manifold} \label{qua}
 
The following quadrature formula holds on $E_{k}(\bf M)$, see Theorem 5.3 in \cite{GP10}. For every $k\in\mathbb{N}$ there exists a finite subset $\chi_{k}$ of $\bf M$ of cardinality $|\chi_k| \le C k^{d/2}$ and positive real numbers $b_{\eta}:=b_{\eta k}>0$, indexed by the elements $\eta$ of
$\chi_{k},$ such that
\begin{equation}
\forall f\in E_{k}(\mathbf M),\qquad\int_{\mathbf M} f(x)\,dx=\sum_{\eta\in\mathcal{X}_{k}}b_{\eta}f(\eta)\ .  \label{quadr}
\end{equation}
The kernel $C_{j}$ defined above clearly satisfies $z\mapsto C_{j}(x,z)\in E_{2^{2j}}(\mathbf M)$ for every $x \in \mathbf M$, and Theorem 6.1 in \cite{GP10} states that 
\begin{equation} \label{add}
f,g \in E_n(\mathbf M) \Rightarrow fg \in E_{4\tau n}(\mathbf M),
\end{equation}
so we deduce $z\mapsto C_{j}(x,z)C_{j}(z,y)\in E_{\tau 2^{2j+2}}(\mathbf M)$. Note that it is property (\ref{add}) where homogeneity of the manifold is used crucially. It is in the same spirit as (but not equivalent to) the addition formula for eigenfunctions of the Laplace-Beltrami operator on a Riemannian manifold (see \cite{G75}). Combining (\ref{aj}) with (\ref{quadr}) thus implies
\begin{equation*}
A_{j}(x,y)=\int_\mathbf M C_{j}(x,z)C_{j}(z,y)dz=\sum_{\eta\in\chi
_{\tau 2^{2j+2}}}b_{\eta}C_{j}(x,\eta)C_{j}(\eta,y)\ 
\end{equation*}
and the action of $A_j$ on $L^2(\mathbf M)$ can hence be represented as
\begin{align*}
A_{j}f(x) & =\int_\mathbf M A_{j}(x,y)f(y)dy=\int_\mathbf M \sum_{\eta \in\chi_{\tau 2^{2j+2}}}b_{\eta}C_{j}(x,\eta)C_{j}(\eta,y)f(y)dy \\
& =\sum_{\eta\in\chi_{\tau 2^{2j+2}}}\sqrt{b_{\eta}}C_{j}(x,\eta)\int_{\mathbf M} {\sqrt{b_{\eta}}C_{j}(\eta,y)}f(y)dy.
\end{align*}
This motivates the definition of the \textit{needlet scaling function} $\phi_{j \eta}$ indexed by the cubature points $\eta \in \mathcal Z_j$, 
\begin{equation*}
\phi_{j\eta}(x):=\sqrt{b_{\eta}}\,C_{j}(x,\eta); \qquad \eta\in\mathcal
{Z}_j \equiv \chi_{\tau 2^{2j+2}}.
\end{equation*}
With this notation we can write 
\begin{equation} \label{needpro}
A_jf(x)= \sum_{\eta \in \mathcal Z_j} \langle \phi_{j \eta},f \rangle \phi_{j \eta} (x),
\end{equation}
and call this approximation the \textit{needlet projection} of $f$ onto $E_{\tau 2^{2j+2}}(\bf M)$ at resolution level $j$. 

We shall need below the following estimates on the cubature set, see \cite{GP10}
\begin{equation}\label{card}
\quad\tfrac 1{k_1}\frac 1{2^{dj}}\le b_{\eta j} \le k_1  \frac 1{2^{dj}}~~~\forall \eta \in \mathcal Z_j, \quad\tfrac 1{k_2} \, 2^{dj}\le  |\mathcal {Z}_j| \le {k_2}\, 2^{dj}
\end{equation}
for some explicit constants $k_1, k_2>0$.

Although we shall not explicitly use it in what follows, we can telescope the needlet projections in the usual way to obtain a wavelet-type multiresolution approximation $$A_jf= A_0f + \sum_{0 \le l \le j-1} \sum_\eta \langle f, \psi_{l \eta}\rangle \psi_{l\eta}$$ of a function $f$ on a compact homogeneous manifold by needlets $$\psi_{l \eta}(x) = \sqrt{b_{\eta l}} \sum_m c(\lambda_m/2^{2l})L_k(x, \eta), ~~\eta \in \mathcal Z_l,$$ with $c(y)=\sqrt{a(y/2)-a(y)}$. See Section 8 of \cite{GP10} for details. In particular
\begin{equation*}
f \in L^2(\mathbf M) \Rightarrow \left\|f-  \sum_{l\leq j} \sum_{\eta \in  \mathcal Z_l}  \langle f , \psi_{l \eta}  \rangle  \psi_{l \eta},\right\|_2 \to 0 ~\text{ as } j \to \infty,
\end{equation*}
and the $(\psi_{j \eta})$'s form a tight frame of $L^2(\bf M)$:
\begin{equation}\label{FRA}
 \forall  f \in L^2(\mathbf M), \quad \|f\|_2^2 = \sum_j \sum_{\eta} | \langle f , \psi_{j \eta} \rangle |^2.
\end{equation}

\subsection{Properties of the Needlet Frame}

We establish some key properties of needlets, including their near-exponential localization property.
 
\begin{prop} \label{needp}We have, for some constant $0<D_1(\mathbf M)<\infty$ and every $j \ge 0, \eta \in \mathcal Z_j,$
\begin{equation} \label{subonb}
 \|\phi_{j \eta}\|_2 \le 1,  \quad \|\phi_{j \eta}\|_\infty \le D_1(\mathbf M) 2^{jd/2}.
\end{equation}
Moreover, for every $x \in \mathbf M, \eta \in \mathcal Z_j$ and every $N \in \mathbb N$ there exists a constant $c_N$ such that 
\begin{equation} \label{ndl}
|\phi_{j \eta}(x)| \le \frac{c_N 2^{jd/2}}{(1+2^{jd}d_\mathbf M(\eta,x))^N}.
\end{equation}
\end{prop}
\begin{proof} 
For the first inequality in (\ref{subonb}), let $\eta \in \bf M$, $n \in \mathbb N$ and note $$\int_{\mathbf M} \left(\sum_{k: \lambda_k \leq n} L_k(x, \eta)\right)^2 dx = \sum_{k: \lambda_k \leq n}  L_k(\eta, \eta).$$ On the other hand $$x \mapsto \left(\sum_{k: \lambda_k \leq n} L_k(x, \eta)\right)^2 \in E_{4\tau n}(\mathbf M),$$ so if $\chi_{4\tau n}$ is the set of cubature points of $E_{4 \tau n}(\mathbf M)$ and $\eta \in \chi_{4\tau n}$
\begin{eqnarray*}
\int_\mathbf M \left(\sum_{k: \lambda_k \leq n} L_k(x, \eta)\right)^2 dx = \sum_{\xi \in \chi_{4\tau n}} b_\xi \left(\sum_{k: \lambda_k \leq n} L_k(\xi, \eta)\right)^2 \ge b_\eta \left(\sum_{k: \lambda_k \leq n} L_k(\eta, \eta)\right)^2.
\end{eqnarray*}
so, combining these estimates,
$$b_\eta \le \frac{1}{\sum_{k: \lambda_k < n} L_k(\eta, \eta)}$$
for every $\eta \in \chi_{4dn}$. This implies, for every $\eta \in \mathcal Z_j$,
\begin{eqnarray*}
\int_{\mathbf M} \phi^2_{j \eta}(x)dx = b_\eta \sum_{k: \lambda_k < 2^{2j}} a (\lambda_k/2^{2j}) L_k(\eta, \eta) \le 1.
\end{eqnarray*}
To prove the remaining claims, recall that by definition $$\phi_{j \eta}(x) = \sqrt {b_\eta} \sum_{k: \lambda_k < 2^{2j}} \sqrt {a (\lambda_k/2^{2j})} L_k(x,y).$$ 
For $f$ a function from the Schwartz-class on $\mathbb R^+$, Lemma 4.1 (and the remark after it) in \cite{GM09},  applied to the elliptic operator $f(\mathcal L/2^{2j})$ (notation of functional calculus, $t=2^{-2j}$ in their lemma), proves that for every integer $N \ge 0$ there exists a constant $c_N(f)$ such that 
 \begin{equation}\label{GM}
 \sum_{k: \lambda_k < 2^{2j}} f (\lambda_k/2^{2j}) L_k(x,\eta) \le \frac{c_N(f) 2^{jd}}{( 1+2^{jd}d(\eta,x))^N}.
 \end{equation} 
Applying this to $f= \sqrt{a},$  we infer the second bound in (\ref{subonb})  and (\ref{ndl}) follows from (\ref{card}) and (\ref{GM}).
\end{proof}


%

\medskip


%
\begin{prop} \label{weyl}
We have 
\begin{equation}
\sup_{x \in \mathbf M} \int_\mathbf M A^2_j(x,y)dy \le D_2(\mathbf M) 2^{jd},~~~ \sup_{x,y \in \mathbf M}|A_j(x,y)|\le D_2(\mathbf M) 2^{jd}
\end{equation}
for some finite positive constant $D_2(\mathbf M)$ that depends only on the manifold.
\end{prop}
\begin{proof}
As  $A_j(x,y):=  \sum_{k} a(\lambda_k/2^{2j})  L_k(x,y)$, the second claim follows from (\ref{GM}) with $f=a$. For the first 
$$\int_\mathbf M A^2_j(x,y)dy= \int_\mathbf M  \sum_{k,l} a\left(\frac{\lambda_k}{2^{2j}}\right)  L_k(x,y)  \overline{  a\left(\frac{\lambda_l}{2^{2j}}\right)   L_l(x,y)}   dy$$
$$=   \sum_{k}a^2\left(\frac{\lambda_k}{2^{2j}}\right)  L_k(x,x) $$
and again using (\ref{GM}) with $f=a^2$ gives the result.
%
\end{proof}

\subsection{The case of $\bS^d$}
 
In the case of the $d$-dimensional unit sphere $\mathbb S^d$ of $\mathbb R^{d+1}$ the above construction is effectively the one in \cite{NPW}. On $\mathbb S^d$ the differential operator $\mathcal L$ coincides with the usual Laplace-Beltrami operator, and we have
$$ L^2(\mathbb S^d) = \bigoplus_k \cH_k, \quad \cH_k \equiv \cH_k(\mathbb S^d) =ker (\Delta-\lambda_k I),  \;  \lambda_k= -k(k+d-1).$$
The eigenfunctions $e_k$ in this case are the spherical harmonics with eigenvalues $k(k+d-1)$ (e.g., Proposition 9.3.5 in \cite{Faraut}). Thus if we take the subsequence $N\equiv N_k$ of $\mathbb N$ for which $k(k+d-1)=N_k$ as $k$ runs through the nonnegative integers, then the spaces $E_N(\mathbb S^d)$ correspond to the spaces $\mathcal P_N(\bS^d)$ of spherical polynomials of degree less than or equal to $N$, which are spanned by the mutually orthogonal spaces $\cH_k(\bS^d), 0 \le k \le n,$ of spherical harmonics, see \cite{Faraut, STW}. 

If  $\{e_i^k\}$ is any orthonormal basis of $\cH_k,$ then we write, in slight abuse of notation, $$L_k(x,y) = \sum_i e_i^k(x) \overline{e_i^k(y)}= L_k(\langle x,y\rangle_{d+1}) , \quad  \langle x,y\rangle_{d+1}= \sum_{i=1}^{d+1} x_i y_i$$ $$ |\bS^d| L_k(u) = \left(1+ \frac k\nu\right) C^\nu_k (u), \quad \nu = \frac{d-1}2, \quad u\in[-1,1]$$ where $C_k^\nu$ is the corresponding Gegenbauer polynomial, and $|\bS^d|$ is the Lebesgue measure of $\bS^d$, i.e., $|\bS^d| = \int_{\bS^d}dx = (2\pi^{(d+1)/2})/\Gamma( (d+1)/2).$ We have furthermore (p.144 in \cite{STW}) for every $x \in \bS^d,$ $$\sum_i|e^k_i(x)|^2dx = \frac{dim(\cH_k(\bS^d))}{|\bS^d|}$$ and thus 
\begin{equation} 
\label{ldim} |\mathbb S^d| L_k(1) = \dim(\cH_k(\bS^d)).
\end{equation} 
Moreover, for $d \ge 2$ and any $n \in \mathbb N$, $\cP_n(\bS^d) = \bigoplus_{k=0}^n \cH_k(\bS^d)$ and as a consequence, by \cite{STW}, $$dim(\cH_k(\bS^d)) =   C_{k+d}^d-  C_{k-2+d}^d  =  \frac{(d+k-2)! (d+2k-1)}{  k! (d-1)!}$$$$dim(\cP_n(\bS^d) )= C_{n+d}^d+  C_{n+d-1}^d= \frac 2{d!}(n+1)(n+2)..(n+d-1)(n+ \frac d2)=$$ $$\frac 2{d!} n^d \left(1+\frac 1n\right)\left(1+\frac 2n\right) ....\left(1+\frac{d- 1}n\right)\left(1+\frac d{2n}\right) =n^d \left(\frac 2d+\frac 1n \right)   \prod_{j=1}^{d-1} \left(\frac 1j+ \frac 1n \right).$$ So, for $d \geq 2,  n\geq 2$, $$\frac 2{d!} (n+ 1)^d  \leq   dim(\cP_n(\bS^d) )\leq n^d \left(\frac{n+1}n\right)^2 $$ 
$$\frac 2{d!} n^d  \leq   dim(\cP_{n-1}(\bS^d) )\leq n^d ~~\text{ and }~~ dim(\cP_{n-1}(\bS^1))= n.$$
By virtue of these bounds the constants in Proposition \ref{weyl} can be explicitly calculated. To obtain a unified notation define, for $j \in \mathbb N$, the integers $k(j)=\max\{k \in \mathbb N: \lambda_k=k(k+d-1) <2^{2j} \}$ so that $k(j)<2^j$ always holds. Then
\begin{eqnarray*}
\int_{\mathbb S^d} A^2_j(x,y)dx &=& \sum_{k: \lambda_k < 2^{2j}}  [a(\lambda_k/2^{2j})]^2  L_k(1) =  \frac 1{|\bS^d|}\sum_{k: \lambda_k <2^{2j}}    [a(\lambda_k/2^{2j})]^2  dim (\cH_k(\bS^d))  \\
&\leq& \frac{dim (\cP_{k(j)}(\bS^d))}{|\bS^d|}   \le \frac {2^{jd}}{|\bS^d|},
\end{eqnarray*}
and these inequalities imply that the same bound holds for $|A_j(x,y)|$. We can also deduce, as in the proof of Proposition \ref{needp}
$$ \| \phi_{j\eta}\|_\infty = \sqrt{b_\eta}  \sum_{k: \lambda_k < 2^{2j}} \sqrt{ a(\lambda_k/2^{2j})}  L_k(1) \leq  \sqrt{\sum_{k: \lambda_k < 2^{2j}} L_k(1)} \le  \sqrt{\frac{2^{jd}}{|\bS^d|}} .$$ Conclude that the key constants $D_1(\mathbf M), D_2(\mathbf M)$ in the last subsection can be taken to be 
\begin{equation}\label{const}
D_1(\mathbb S^d) = \sqrt{\frac{1}{|\bS^d|}},\quad D_2(\mathbb S^d) = \frac{1}{|\mathbb S^d|}
\end{equation}
in the case of the unit sphere. Finally we should remark that in the case of the unit sphere the addition formula (\ref{add}) holds with $4\tau n$ replaced by $2n$  as one is multiplying spherical polynomials. [Indeed whenever the Laplace-Beltrami operator coincides with $\mathcal L$ one can use the addition formula for eigenfunctions of the Laplacian in \cite{G75}.] Moreover, if $d=2$, for each resolution level $j$, the HEALPix pixelisation (commonly used for astrophysical data) gives $12 \cdot 2^{2j}$ cubature points, so $k_2=12$ in (\ref{card}).

\section{Linear Needlet Density Estimators and Concentration Properties of their Uniform Fluctuations} \label{linsup}

Let $X,X_1,...,X_n$ be i.i.d.~random variables taking values in a compact homogeneous manifold $\bf M$ of dimension $d$. Denote their common law by $P$ and assume that $P$ possesses a density $f: \mathbf M \to [0, \infty)$ w.r.t.~$dx$ on $\bf M$. Denote further by $P_n= \frac{1}{n}\sum_{i=1}^{n} \delta_{X_i}$ the empirical measure of the sample. Let $A_j(x,y)$ be the needlet projection kernel. For $j \in \mathbb N$, the linear needlet density estimator of $f$ is defined as
\begin{equation} \label{est}
f_n(j,y) = \frac{1}{n} \sum_{i=1}^n A_j(X_i,y) = \int_{\mathbf M} A_j(x,y)dP_n(x), ~~~~~~~ y \in \mathbf M.
\end{equation}
We shall often write, in slight abuse of notation, $f_n(j)$ for $f_n(\cdot,j)$.

\subsection{A Bernstein-type Concentration Inequality for Needlet Estimators}

We define now some quantities that measure the 'Gaussian' and 'Poissonian' fluctuations of the uniform deviations of the centered estimator $f_n(j)$. Recall the explicit constants $D_1(\mathbf M), |\mathcal Z_j| \le k_2 2^{jd}$ from (\ref{card}), (\ref{subonb}) in the previous section. Note moreover that the second estimate in Proposition \ref{needp} immediately implies  
\begin{eqnarray} \label{cnull}
2^{jd/2}c_0(\mathbf M,j) \equiv \sup_{x \in \mathbf M} \sum _{\eta \in\mathcal {Z}_j }| \phi_{j\eta}(x)| \le 2^{jd/2}C(\mathbf M).
\end{eqnarray} 
The constant $c_0(\mathbf M,j)\equiv c_0(\mathbf M, j, a, \mathcal Z_j)$ (or an upper bound for it) can be computed explicitly after the regularizing function $a$ and the quadrature set $\mathcal Z_j$ have been chosen, and a sharp numerical evaluation of it is important in application of Proposition \ref{pisneed} below. 

Define then
\begin{eqnarray*} \label{sigma}
\bar \sigma(n,l,x)&:=& \bar \alpha(x,l)\sqrt{\frac{2^{ld}}{n}} +  \bar \alpha'(x,l)\frac{2^{ld}}{n}
\end{eqnarray*}
where $$\bar \alpha(x,l):=\bar \alpha(\mathbf M,f,x,l):=c_0(\mathbf M,l)\sqrt{2(\log (2|\mathcal Z_l|)+x) \|f\|_\infty}$$ and $$\bar \alpha'(x,l):=\bar \alpha'(\mathbf M,x,l):= c_0(\mathbf M,l)\frac{2}{3} D_1(\mathbf M)(\log (2|\mathcal Z_l|)+x).$$
We now prove the following concentration inequality for the needlet density estimator. 
\begin{prop} \label{pisneed}
Let $\mathbf M$ be a compact homogeneous manifold and suppose $f: \mathbf M \to [0,\infty)$ is bounded. We have, for every $n \in \mathbb N$, every $j \in \mathbb N$ and every $x \ge 0$
\begin{equation*}
\Pr \left \{\sup_{y \in \mathbf M} \left|f_n(j,y)-Ef_n(j,y)\right| \ge \bar \sigma(n,j,x) \right\} \le e^{-x}.
\end{equation*}
\end{prop}
\begin{proof}
The explicit cubature formula for eigenfunctions of $\mathcal L$ allows to reduce the infinite supremum $\sup_{y \in \mathbf M} |f_n(j,y)-Ef_n(j,y)|$ to one over a finite set, so that finite-dimensional probabilistic methods can be applied. Indeed, the estimate (\ref{cnull}) implies that the supremum of any $h \in E_{2^{2j-1}}(\mathbf M)$ over $\bf M$ can be bounded by the (finite) maximum of the needlet coefficients of $h$: Clearly from (\ref{polyid}) $$\forall h \in E_{2^{2j-1}}(\mathbf M), \quad h(x) = A_jh(x)=\sum _{\eta \in\mathcal {Z}_j } \langle \phi_{j\eta}, h \rangle \phi_{j\eta}(x)$$
so that for $\mathcal Z_j$ a cubature set of $E_{\tau 2^{2j+2}}(\bf M)$ one has
\begin{equation} \label{supimbed}
\sup_{x \in \mathbf M} \left| h(x) \right| \le  \max_{\eta \in\mathcal Z_j} \left|\langle \phi_{j\eta}, h \rangle \right|  \sup_{x \in \mathbf M} \sum _{\eta \in\mathcal {Z}_j }| \phi_{j\eta}(x)| = 2^{jd/2}c_0(\mathbf M,j)  \max_{\eta \in\mathcal Z_j} \left|\langle \phi_{j\eta}, h \rangle \right|.
\end{equation}
Now using $\langle \cdot, \cdot \rangle$ notation also acting on finite signed measures, $$\|f_n(j)-Ef_n(j)\|_\infty = \sup_{y \in \mathbf M}\left|\sum_{\eta \in \mathcal Z_j} \phi_{j\eta}(y) \langle \phi_{j \eta}, P_n-P \rangle \right| \le 2^{jd/2} c_0(\mathbf M,j)  \max_{\eta \in \mathcal Z_j} \left|\langle \phi_{j\eta}, P_n-P\rangle \right|$$ by (\ref{cnull}) above. Consider the finite empirical process indexed by the class of functions $\{\phi_{j\eta_k}\}_{k=1}^{|\mathcal Z_j|}$ which has envelope $U= 2^{jd/2}D_1(\mathbf M)$ in view of (\ref{subonb}). The class of functions $$\mathcal G := \left\{\phi_{j \eta_1}/2U,...,\phi_{j \eta_{|\mathcal Z_j|}}/2U \right\},$$ is thus uniformly bounded by $1/2$ and its weak variances $\sigma^2$ satisfy $$\sup_{g \in \mathcal G}Eg^2(X) \le \sigma^2=\frac{\|f\|_\infty}{2^{jd+2}D^2_1(\mathbf M)}$$ since $\|\phi_{j \eta}\|_2\le 1$ (again (\ref{subonb})). Recall Bernstein's inequality (e.g., p.26 in \cite{Massart}): If $Z_1,...,Z_n$ are i.i.d.~centered random variables bounded in absolute value by $1$ then
\begin{equation} \label{bernstein}
\Pr \left\{\left|\frac{1}{n}\sum_{i=1}^n Z_i\right| \ge \sqrt{\frac{2tv}{n}} + \frac{t}{3n} \right\} \le 2e^{-t}
\end{equation}
where $v \ge EZ_i^2$. Therefore, using the notation $\|\mu\|_\mathcal G \equiv \sup_{g \in \mathcal G}|\int gd\mu|$ for signed measures $\mu$,
\begin{eqnarray*}
&& \Pr \Bigg\{\|f_n(j)-Ef_n(j)\|_\infty \ge c_0(\mathbf M,j)\bigg( \sqrt{\frac{2(\log (2|\mathcal Z_j|)+x) 2^{jd} \|f\|_\infty}{n}} \\
&&~~~~ \quad ~~ \quad ~~~~~~~~~~~~~~~~~~~ \quad+ \frac{2U2^{jd/2}(\log (2|\mathcal Z_j|)+x)}{3n}\bigg) \Bigg\} \\
&&\le \Pr \left \{\max_{\eta \in \mathcal Z_j} \left|\langle \phi_{j,\eta}, P_n-P\rangle \right| \ge \sqrt{\frac{2(\log (2|\mathcal Z_j|)+x) \|f\|_\infty}{n}} + \frac{2U(\log (2|\mathcal Z_j|)+x)}{3n}  \right\} \\
&& \le \Pr\left\{\|P_n-P\|_\mathcal G \ge \sqrt{\frac{2(\log (2|\mathcal Z_j|)+x) \|f\|_\infty}{D^2_1(\mathbf M)2^{jd+2} n}} + \frac{\log (2|\mathcal Z_j|)+x}{3n} \right\} \\
&& = \Pr \left\{\max_{m=1,...,|\mathcal Z_j|}\left|\frac{1}{n}\sum_{i=1}^n (g_m(X_i)-Eg_m(X))\right|  \ge \sqrt{\frac{2(\log (2|\mathcal Z_j|)+x)\sigma^2}{n}} + \frac{\log (2|\mathcal Z_j|)+x}{3n} \right\} \\
&& \le \sum_{m=1}^{|\mathcal Z_j|} \Pr \left\{\left|\frac{1}{n}\sum_{i=1}^n (g_m(X_i)-Eg_m(X))\right|  \ge \sqrt{\frac{2(\log (2|\mathcal Z_j|)+x)\sigma^2}{n}} + \frac{\log (2|\mathcal Z_j|)+x}{3n} \right\} \\
&& \le 2 |\mathcal Z_j| \exp\left\{-\log (|\mathcal Z_j|)-\log 2 -x)\right\} =e^{-x},
\end{eqnarray*}
which completes the proof of Proposition \ref{pisneed}.
\end{proof}

\medskip

We should mention that a minor modification of the proof of Proposition \ref{pisneed} combined with the usual blocking arguments (as, e.g., in Theorem 1 in \cite{GineNickl9a}) implies under standard conditions on $j_n$ (including $2^{j_n} \approx n^\eta$ for some $0<\eta<1$) that 
\begin{equation}
\limsup_n \sqrt{\frac{2^{j_nd}j_n}{n}}\sup_{y \in \mathbf M}|f_n(j,y)-Ef_n(y,d)| \le D \quad almost~surely
\end{equation}
where the constant $D$ depends only on $\mathbf M, k_2$ and $\|f\|_{\infty}$. 

In some proofs below we shall need that $\bar \sigma(n,l,x)$ is monotone increasing in $l \in \mathbb N$. In general whether this holds true or not depends on the cubature $\mathcal Z_l$ as well as on the function $a$. Monotonicity of $\bar \sigma(n,l,x)$ can be easily ensured if we replace $\bar \alpha(x,l)$ and $\bar \alpha'(x,l)$ by their upper bounds $\alpha(x,l), \alpha'(x,l)$ obtained from the inequalities $|\mathcal Z_l| \le k_2 2^{ld}, c_0(\mathbf M,l) \le C(\mathbf M)$. While we do not advocate this in practice, for the theoretical development we define 
\begin{equation} \label{AAA}
\sigma(n,l,x) = \alpha(x,l)\sqrt{\frac{2^{ld}}{n}} +  \alpha'(x,l)\frac{2^{ld}}{n}, ~~~ A(n,l,x):=\left[\alpha(x,l) + \alpha'(x,l)\sqrt{2^{ld}/n} \right].
\end{equation}
The constant $A(n,l,x)$ allows for $\sigma(n,l,x)$ to be written as a constant multiple of the 'Gaussian component' $\sqrt{2^{ld}/n}$, that is, $\sigma(n,l,x) = A(n,l,x) \sqrt{2^{ld}/n}.$

\subsection{Concentration Inequalities via Rademacher Processes on Manifolds}

Despite its conceptual simplicity the approach from the previous section has one drawback: the uniform deviations of $f_n-Ef_n$ are controlled globally on $\mathbf M$ by the function $\sigma(n,l,x)$ -- constant on $\mathbf M$. For functions $f$ that exhibit spatially inhomogeneous regularity properties it is of interest to have a 'localised' version of $\sigma(n,l,x)$. This could be achieved in Proposition \ref{pisneed} by means of proving a 'local' analogue of (\ref{supimbed}), which, however, is a rather intricate matter that we do not pursue here. Instead we show how a simple symmetrization technique can be used to deal with this problem. This is inspired by \cite{KolaAOS} and also \cite{GineNickl10b}. For $\Omega$ any subset of $\mathbf M$, define a Rademacher process $\{(1/n)\sum_i \varepsilon_i A_j(X_i,y)\}_{y \in \Omega}$ and set 
$$R_n(\Omega,j)=\sup_{y \in \Omega} \left |\frac{1}{n} \sum_{i=1}^{n} \varepsilon_i A_j(X_i,y) \right|$$ 
with $(\varepsilon_i)_{i=1}^n$ an i.i.d.~Rademacher sequence, independent of the $X_i$'s (and defined on a large product probability space). $R_n(\Omega, j)$ can be computed in practice by first simulating $n$ i.i.d.~random signs, applying these signs to the summands $A_j(X_i)$ of the needlet density estimator, and maximizing the resulting function. The idea is that the supremum $R_n(\Omega, j)$ of the symmetrized process serves as a random surrogate for the unknown supremum $\sup_{y \in \Omega}|f_n(j,y)-Ef_n(j,y)|$ of the centered process. Indeed Proposition \ref{radneed} shows that $\sup_{y \in \Omega}|f_n(y)-Ef_n(y)|$ concentrates around (a constant multiple of) $R_n(\Omega,j)$. Define the deviation term
 $$\sigma^{R}(\Omega, n,j,x)=6R_n(\Omega, j)+10\sqrt{\frac{2^{jd}D_2(\mathbf M)\|f\|_\infty (x+ \log 2)}{n}}+22 \frac{2^{jd} D_2(\mathbf M)(2x+2\log2)}{n}.$$
 
\begin{prop} \label{radneed}
Let $\mathbf M$ be a compact homogeneous manifold and suppose $f:\mathbf M \to [0,\infty)$ is bounded. We have for every $n \in \mathbb N$, every $j \in \mathbb N$, every $\Omega \subseteq \mathbf M$ and every $x>0$ that 
\begin{equation*}
\Pr \left \{\sup_{y \in \Omega} \left|f_n(y,j)-Ef_n(y,j)\right| \ge \sigma^R(\Omega, n,j,x) \right\} \le e^{-x}.
\end{equation*}
\end{prop}
\begin{proof}
We use the following general result for empirical processes. 
\begin{prop} \label{rad} Let $\mathcal F$ be a countable class of real-valued measurable functions defined on $\mathbf M$, uniformly bounded by $1/2$. We have for every $n \in \mathbb N$ and $x>0$
\begin{equation*}
\Pr \left \{\left\|\frac{1}{n}\sum_{i=1}^n (f(X_i)-Pf)\right\|_\mathcal F \ge 6\left\|\frac{1}{n}\sum_{i=1}^n \varepsilon_i f(X_i)\right\|_\mathcal F + 10\sqrt{\frac{(x+\log 2)\sigma^2}{n}} + 22 \frac{x+\log 2}{n} \right\} \le e^{-x}
\end{equation*}
\end{prop}
The proof, which is based on \cite{tala}'s inequality with constants (e.g., \cite{Massart}), is inspired by ideas in \cite{KolaAOS}, \cite{GineNickl10b}, and can be found in Proposition 5 in \cite{ln10}. Now to prove Proposition \ref{radneed} note that
\begin{equation*}
\|f_n(j)-Ef_n(j)\|_\Omega = \sup_{y \in \Omega} \left|\frac{1}{n} \sum_{i=1}^n \left(A_j(X_i,y)-EA_j(X,y)\right)\right| 
\end{equation*}
for $\Omega \subseteq \mathbf M$. This amounts to studying the empirical process indexed by the class of functions $\left\{A_j(\cdot,y): y \in \Omega \right\}$ for $\Omega \subseteq \mathbf M$. 
This class has envelope $2^{jd}D_2(\mathbf M)$ in view of Proposition \ref{weyl}. Define thus
\begin{equation} \label{class}
\mathcal G := \mathcal G_j = \left \{A_j(\cdot,y)/(2^{jd+1}D_2(\mathbf M)): y \in \Omega \right\}
\end{equation}
which is uniformly bounded by $1/2$. [In fact, by continuity of the mapping $y \mapsto A_j(x,y)$ for every $x \in \mathbf M$ we can restrict ourselves to a countable subset of $\Omega$, which we still denote by $\Omega$.] Furthermore the upper bound for the weak variances can be taken to be 
\begin{equation} \label{weak0}
\sup_{g \in \mathcal G}Eg^2(X)  \le \frac{\|f\|_\infty }{D_2^2(\mathbf M)2^{2jd+2}} \sup_{y \in \mathbf M}\int_{\mathbf M} A_j^2(x,y)dx \le \frac{\|f\|_\infty}{D_2(\mathbf M)2^{jd+2}} =:\sigma^2
\end{equation}
in view of Proposition \ref{weyl}. Then, recalling the notation $\|\cdot\|_\mathcal G$ from the proof of Proposition \ref{pisneed}
$$\Pr \Bigg\{\left\|f_n(j,\cdot)-Ef_n(j,\cdot)\right\|_\Omega \ge 6R_n(\Omega, j) + 10\sqrt{\frac{2^{jd}D_2(\mathbf M)\|f\|_\infty(x+\log 2)}{n}} $$ $$~~~~~~~~~+ 22 \frac{2^{jd}D_2(\mathbf M)(2x+2\log 2)}{n} \Bigg\} $$
$$= \Pr \left \{\left\|\frac{1}{n}\sum_{i=1}^n (g(X_i)-Pg)\right\|_\mathcal G \ge \frac{6R_n(\Omega, j)}{2^{jd+1}D_2(\mathbf M)} + 10\sqrt{\frac{\|f\|_\infty(x+\log 2)}{D_2(\mathbf M)2^{jd+2}n}} + 22 \frac{x+\log 2}{n} \right\} $$
$$= \Pr \left \{\left\|\frac{1}{n}\sum_{i=1}^n (g(X_i)-Pg)\right\|_\mathcal G \ge 6\left\|\frac{1}{n}\sum_{i=1}^n \varepsilon_i g(X_i)\right\|_\mathcal G  + 10\sqrt{\frac{(x+\log 2)\sigma^2}{n}} + 22 \frac{x+\log 2}{n} \right\}$$
and the last expression is less than or equal to $e^{-x}$ using Proposition \ref{rad} with $\mathcal G$ as in (\ref{class}) and $\sigma$ specified by (\ref{weak0}).
\end{proof}

\medskip

It is interesting to compare $\sigma^R$ to $\sigma$ from Proposition \ref{pisneed}. On the one hand the second and third terms defining $\sigma^R(\Omega, n,j,x)$ are of a smaller asymptotic order than $\sigma(n,j,x)$ for $j \to \infty$ due to the absence of $|\mathcal Z_j|$ in $\sigma^R$. On the other hand the term $R_n(\Omega, j)$ is random, and one is led to ask whether in average $\sigma^R$ will be larger or smaller than $\sigma$. Our proofs imply, for some constant $C$ independent of $j,n$, that $$ER_n(\Omega, j) \le C \left( \sqrt{\frac{2^{jd}j}{n}} + \frac{2^{jd}j}{n}\right)$$ so that $\sigma^R$ has the same size as $\sigma$ as a function of $j,n$, up to constants. 

Inspection of the proofs and arguments similar to those in the proof of Proposition 2 in \cite{GineNickl10b} show that $R_n(\Omega,j)$ in Proposition \ref{radneed} can be replaced by its (conditional) expectation $E_\varepsilon R_n(\Omega, j)$ -- a quantity that may be more stable in applications. Moreover, the constants appearing in the definition of $\sigma^R$ may still be fairly conservative: the proof is based on an application of \cite{tala}'s inequality with explicit constants (see \cite{Massart}), and in the lower deviation version thereof the optimal constants are not known yet. 

\section{Confidence Bands}

If the size of the bias $\|Ef_n(j)-f\|_\infty$ were known, one could directly use Propositions \ref{pisneed} or \ref{radneed} and a suitable choice of $j$ to obtain confidence bands with prescribed finite sample coverage. For instance, if $f$ is the uniform distribution (volume element) on $\mathbf M$, the bias $A_0(f)-f$ of the estimate $f_n(0)$ is exactly zero. In analogy, if $f \in E_n(\mathbf M)$ is a finite linear combination of eigenfunctions of $\mathcal L$ (so in the spherical case a polynomial) then the estimator $f_n(J)$ for sufficiently large but finite $J$ also has bias zero (cf.~(\ref{polyid})). As usual, going beyond finite-dimensional smoothness classes is possible by considering spaces of differentiable functions on $\mathbf M$. For instance one defines
$C^k( \mathbf M)$ as the set of continuous functions  $f \in C(\mathbf M)$ such that for all $X_1, X_2,\ldots, X_k$ in  $Lie(G),\; D_{X_1} D_{X_2} \ldots D_{X_k} f \in C( \mathbf M).$ It is a Banach space when equipped with the following norm:
$$ \|f \|_{C^k} = \sup_{| X_1| \leq 1,\ldots,| X_k| \leq 1} \|D_{X_1} D_{X_2} \ldots D_{X_k} f \|_\infty + \|f \|_\infty,$$ and $C^\infty (\mathbf M)$ is the intersection of all the spaces $C^k(\mathbf M), k \in \bN.$ One can define such spaces also for noninteger $k$ by introducing a modulus of continuity along vectorial directions $X$, and the resulting scale of H\"{o}lder-Zygmund function spaces $\mathcal C^k(\mathbf M)$ can be characterized by the decay of their needlet coefficients in very much the same way as in the case of H\"{o}lder-Zygmund spaces on Euclidean spaces: A $k$-regular function in $\mathcal C^k(\mathbf M), k>0$ then satisfies the estimate 
\begin{equation} \label{44}
\|A_j(f) - f\|_\infty \le C 2^{-jk}.
\end{equation}
See \cite{GP10} for these results. If the smoothness degree $t$ of $f$ is known such bounds can be used, together with Propositions \ref{pisneed}, \ref{radneed}, in the construction of asymptotic confidence sets, proceeding in the same way as in the classical paper \cite{bros} via choosing a resolution level $j_n$ that leads to 'undersmoothing', i.e., a bias of smaller order as a function of $n$ than the random fluctuations of the centered estimators.

However, in the typical nonparametric function estimation problem the size of the bias is not known, and the above assumptions are far from realistic. So we have the more ambitious goal to obtain confidence sets for the needlet estimator with an automatic choice of the resolution level $j$.

\subsection{Estimate of the Resolution Level}

Split the sample into two parts $\cS_1$ and $\cS_2$, each of (integer) size $n_1>0$ and $n_2>0$ respectively. For asymptotic considerations we shall require that $n_1/n_2$ is bounded away from zero and infinity as $n \to \infty$. Denote by $$P_{n_1} = \frac{1}{n_1}\sum_{i=1}^{n_1} \delta_{X_i}, ~~and~~ P_{n_2} = \frac{1}{n_2} \sum_{i=1}^{n_2} \delta_{X_{n_1+i}}$$ the empirical measures associated with the first and the second subsample, respectively, and define the associated needlet density estimators $$f_{n_v}(j,y)=\int_{\mathbf M} A_j(x,y)dP_{n_v}(x), ~~~ y \in \mathbf M,~~~v=1,2.$$ 

We use the sample $\cS_2$ to choose the resolution level $j$. For $n_2>1$, choose an integer  $j_{\max}:= j_{\max, n}$ and define the grid of candidate bandwidths as $$\mathcal J := \mathcal J_n = \left\{[0, j_{\max}] \cap \mathbb N \right\}.$$ For asymptotic considerations we shall only require 
\begin{equation} \label{jmax}
2^{j_{\max}} \simeq \left(\frac{n_2}{(\log n_2)^2} \right)^{1/d},
\end{equation}
 but a practical choice is to first choose $l^*$ such that $\alpha(x,l^*)\sqrt{2^{l^*}/n_2}=\alpha'(x,l^*)(2^{l^*}/n_2)$ and to define $j_{\max}$ such that $2^{j_{\max}}=2^{l^*}/(\log n_2)^{1/d}$. Such a choice of $j_{\max}$ is just slightly below the boundary where the Poissonian term starts to dominate the Gaussian term in $\sigma(n_2,l,x)$ in Proposition \ref{pisneed}, and choosing $j>j_{\max}$ would then result in inconsistent estimators, so that $j_{\max}$ is a natural upper bound for $\mathcal J$.

The goal is to select a data-driven bandwidth $\hat j_n$ from $\mathcal J_n$. Heuristically, for $l>j$, $$f_{n_2}(j)-f_{n_2}(l) = [f_{n_2}(j)-Ef_{n_2}(j)] -[f_{n_2}(l)-Ef_{n_2}(l)] + [A_j(f)-f] -[A_l(f)-f]$$ and with large probability the first two terms should not exceed $2\sigma(n_2,l,x)$, a quantity that increases in $l$, and we would like to choose $\hat j_n$ to be the smallest $j$ such that the approximation error $2(A_j(f)-f)$ (which decreases in $j$) does not exceed the size $2\sigma(n_2,l,x)$ of the random fluctuations. 

We shall use the subsample $\cS_2$ to select $\hat j_n$ following this idea, which is due to \cite{Lepski}, formalised as follows:
\begin{equation}
\hat j_n = \min \bigg \{ j \in \mathcal J: \|f_{n_2}(j) - f_{n_2}(l) \|_\Omega \le 4\sigma(n,l) ~~ \forall l>j, l\in {\mathcal J} \bigg \}.
\label{lepski}
\end{equation}where $\sigma(n,l)=\sigma(n_2,l,\kappa{\log n_2})$, cf.~(\ref{AAA}), where $\kappa>0$ is any numerical constant (see Remark \ref{kap} for discussion). By definition $\hat j_n=j_{\max}$ if $\forall j,\; \exists \; l>j,\; l,j \;\in {\mathcal J} $,  $\|f_{n_2}(j) - f_{n_2}(l) \|_\Omega > 4\sigma(n,l)$.

A few remarks about the constants involved in the definition of $\sigma(n,l)$ are in order: All these constants are explicit once the function $a$ and the cubature $\mathcal Z_j$ have been chosen, except for the quantity $\|f\|_\infty$. If no upper bound for $\|f\|_\infty$ is known we advocate that $\|f\|_\infty$ be replaced by $\|f_n(j_{\max})\|_\infty$. Standard arguments imply that this random quantity exponentially concentrates around $\|f\|_\infty$, see for instance \cite{GineNickl10b}. Consequently we neglect the case of $\|f\|_\infty$ unknown in what follows in order to reduce technicalities. Moreover we shall see below how the choice of the numerical constant $\kappa$ influences the finite-sample performance, but our results hold for any choice $\kappa>0$, in particular it does not have to be 'large enough' (as is often assumed in the adaptive estimation literature).

\subsection{Confidence Bands with Random Sizes}

To construct the center of the corridor of the confidence band over $\Omega \subseteq \mathbf M$ we evaluate the linear estimator $f_{n_1}(\cdot,y)$ from (\ref{est}) at the random bandwidth $\hat j_n$. It turns out that some undersmoothing is useful -- in fact crucial -- so let $u_n$ be a sequence of natural numbers and define $$\hat f_n(y) = f_{n_1}(\hat j_n+u_n,y), ~~~~~~~~~~~y \in \Omega.$$ We shall see below how the sequence $u_n$ influences our results but heuristically, and for asymptotic considerations, one may think of $u_n$ of the order $\log \log n$. 

The confidence band we propose is centered at $\hat f_n(y), y \in \Omega,$ and has random size $$s_n(x)=1.01\sigma(n_1,\hat j_n+u_n,x),$$ cf.~(\ref{AAA}), more precisely
\begin{equation} \label{band}
C_n:=C_n(x,y)=\left[\hat f_n (y) - s_n(x), \hat f_n(y) + s_n(x) \right],~~~x>0, ~ y \in \Omega \subseteq \mathbf M.
\end{equation}
Alternatively one can use the band size $s^R_n(\Omega, x) = 1.01 \sigma^R(\Omega, n_1, \hat j_n+u_n, x),$ and all results proved below go through by virtue of Proposition \ref{radneed} and using techniques from Rademacher processes (as in \cite{GineNickl10b}), but we abstain from this to reduce technicalities.

\subsection{Coverage and Adaptation Properties of $C_n$}

\subsubsection{Coverage over Eigenspaces of $\mathcal L$ -- the Finite Dimensional Case}

We first consider here the important case where $f$ is a very smooth function, that is, a fixed linear combination of eigenfunctions of $\mathcal L$, so $f \in E_{2^J-1}(\mathbf M)$ for some fixed $J$. For simplicity of exposition let us consider the case of global confidence bands $\Omega=\mathbf M$ only in this subsection. We start with the case where $f$ equals the volume element of $\mathbf M$.

\begin{theo} \label{uniform} If $f$ is the volume element of $\mathbf M$, $\int_\mathbf M f(x)dx=1$, then we have, for every $n \in \mathbb N$ $$\Pr(\hat j_n = 0) \ge 1- 2j_{\max} n_2^{-2\kappa}.$$ Furthermore, for every $n \in \mathbb N$ and every $x>0$ we have
\begin{equation}
\Pr \left\{f(y) \in C_n(x,y) ~\textrm{ for every }~y \in \mathbf M \right\} \ge 1-e^{-x}
\end{equation}
and, if $2^{u_nd}/n \to 0$ as $n\to \infty$ then $s_n(x) =O_{\Pr} \left(2^{u_nd/2}/\sqrt n \right).$
\end{theo}

In other words our automatic band $C_n$ attains \textit{exact} finite sample coverage if $f$ is uniformly distributed, and in the usual situation where $u_n = \log \log n$ the size of the band shrinks almost at the parametric rate $1/\sqrt n$.

It is instructive to consider next the case where $f \in E_{2^{2J-1}}(\mathbf M) \setminus E_{2^{2J-2}}(\mathbf M)$ for some fixed $J \in \mathbb N$. We would then hope that $\hat j_n =J$ with large probability, as then $A_J(f)-f=0$ (see (\ref{polyid}) above). In the following theorem we restrict ourselves to asymptotic considerations to highlight the main ideas.

\begin{theo} \label{poly}
Suppose $f \in E_{2^{2J-1}}(\mathbf M) \setminus E_{2^{2J-2}}(\mathbf M)$ for some fixed $J \in \mathbb N$. We then have that $$\Pr(\hat j_n \notin [J-1,J]) = O(n^{-2\kappa}+e^{-cn})$$ as $n \to \infty$ for some constant $c$ that depends on $f$ only through $\|f\|_\infty$ and through $$b_1(f)\equiv \inf_{p \in E_{2^{2J-2}}(\mathbf M)}\|p-f\|_\infty>0.$$ Moreover if $u_n>1~ \forall n \in \bN$ then 
\begin{equation} \label{ptc}
\Pr \left\{f(y) \in C_n(x,y) ~\textrm{ for every }~y \in \mathbf M \right\} = 1-e^{-x} - O(e^{-cn})
\end{equation}
 and if $2^{u_nd}/n \to 0$ as $n\to \infty$ then $s_n(x) =O_{\Pr} \left(2^{u_nd/2}/\sqrt n \right).$
\end{theo}

Thus the confidence band $C_n$ has asymptotic coverage for any fixed spherical polynomial, and the asymptotic size of the band $C_n$ is of order $1/\sqrt n$ up to the undersmoothing factor. 

Clearly we have neglected the question of honesty of $C_n$, that is we have not addressed the question whether (\ref{ptc}) holds uniformly in $f \in \cup_{0 \le j \le J-1} E_{2^j}(\mathbf M)$. Inspection of the proof implies that $C_n$ is honest over linear combinations of eigenfunctions of $\mathcal L$ for which the separation constants $b_1(f)$ are bounded below by a constant multiple of $1/\sqrt n$. That uniformity over all densities between $E_{2^{2J-1}}$ and $E_{2^{2J-2}}$ cannot be attained for our 'adaptive' procedure is related to impossibility results for post-model selection estimators in finite-dimensional models, see \cite{lpo}.

\subsubsection{Asymptotic Coverage over H\"{o}lder Balls}

Theorem \ref{poly} just resembles the finite dimensional situation, and if it were indeed known apriori that $f \in E_{2^{J-1}}(\mathbf M)$ for a fixed $J$ one could simply use $f_n(J)$ as an estimator, circumventing the uniformity problems raised in the previous subsection. However if no finite-dimensional model seems realistic for $f$ we may accept these uniformity problems for which $b_1(f)$ is not well-behaved if in return our procedure performs well in the infinite-dimensional setting. Note that in the usual infinite-dimensional nonparametric models the default estimator $f_n(j_{\max})$ has only a logarithmic rate of convergence to zero in supremum norm risk, and will lead to unneccesarily large confidence bands. In contrast our confidence band $C_n$ adapts over an infinite-dimensional class of H\"{o}lder continuous densities $f$ as we show in this section. 

Our first main result is that the size of the band $C_n$ equals, with large probability, the optimal band size that one would obtain from balancing approximation error $A_j(f)-f$ and random fluctuations $f_n(j)-A_j(f)$. For asymptotic considerations this will imply that our band shrinks at the optimal rate of convergence depending on the regularity of $f$. To formalize this statement we shall impose a regularity condition on the density $f$, namely that its approximation errors $\|A_j(f)-f\|_\Omega$ are bounded by a constant multiple of $2^{-jt}$ for some unknown $t>0$. As mentioned in (\ref{44}) above this is tantamount to assuming a classical $t$-H\"{o}lder condition on $f$. The theoretical bandwidth that balances bias and variance is then, up to additive constants (see (\ref{star}) below for an exact definition) $$j_n^*(t) = \frac{1}{2t+d} (\log_2 n-\log_2 \log n).$$

\begin{theo} \label{adapt} {\bf (Size of the band)} Let $\Omega$ be any subset of $\mathbf M$. Suppose $f:\mathbf M \to [0,\infty)$ is bounded and that $\|A_j(f)-f\|_\Omega \le b_2 2^{-jt}$ for some $b_2>0$ and some $t>0$. 
Let $2s_n(x)$ be the diameter of the band $C_n(x,y)$. Then, for every $n \in \mathbb N, x >0$, $$\Pr \left\{s_n(x) > 1.01 \sigma(n_1,j_n^*(t)+u_n+1,x) \right \}\le 2(j_{\max}-j_n^*(t))n_2^{-\kappa }.$$ 
In particular, if the undersmoothing constants $u_n$ are such that $$r_n(t):=\left(\frac{\log n}{n} \right)^{\frac{t}{2t+d}} 2^{\frac{u_nd}{2}}=o(1)$$ as $n \to \infty$ then $s_n(x) = O_{\Pr}(r_n(t))$.
\end{theo}

Note that the proof of the theorem, combined with standard arguments from adaptive estimation (e.g., \cite{GineNickl10b}), implies as well that $\hat f_n$ is rate-adaptive in sup-norm loss, that is, for every $t>0$,
\begin{equation} \label{adaptrisk}
\sup_{f: \|A_j(f)-f\|_\mathbf M \le b_2 2^{-jt}}E\sup_{x \in \mathbf M}|\hat f_n(x)-f(x)|
= O(r_n(t)).
\end{equation}
The rate of convergence $r_n(t)$ cannot be improved over classes of functions that are $t$-H\"{o}lder, see for instance \cite{Kle99} in the case $\mathbf M =\mathbb S^d$, and since these H\"{o}lder classes are, up to constants, sets of the form $\{f: \|A_jf-f\|_\infty \le b_2 2^{-jt}\}$ for suitable $b_2$ (see the results in \cite{GP10}), this implies that (\ref{adaptrisk}) is optimal, and that the band $C_n$ in Theorem \ref{adapt} shrinks at the optimal rate in a minimax sense (up to the undersmoothing factor, which will typically be of size $\sqrt {\log n}$). 

Clearly without a sharp evaluation of the probability of the event $\{f \in C_n\}$ Theorem \ref{adapt} is useless for statistical inference. It is known (see \cite{Lowconf}) that adaptive confidence bands for densities on $\mathbb R$ cannot have coverage over a continuous scale $\bigcup_{t>0}\Sigma(t,b_2)$ of H\"{o}lder balls $\Sigma(t,b_2)$. In a way Low's results can be seen as an infinite-dimensional analogue of the pathologies in finite dimensions mentioned above. On the other hand recent results in \cite{GineNickl10a} show that adaptation is possible over 'generic' subsets of $\bigcup_{t>0}\Sigma(t,b_2)$ when densities are estimated on the real line. The idea is that even if some pathologies cannot be avoided there are still exhaustive classes of densities for which adaptation is possible, and we show how this applies to density estimation on $\mathbf M$. 

To this end we assume the following crucial approximation condition. While the upper bound is standard, the quantity occurring in the lower bound can be viewed as an infinite-dimensional analogue to the constant $b_1$ that appeared in Theorem \ref{poly}. Note that whereas $b_1$ is always positive the lower bound in the following condition may fail to hold for any $t$ for a given continuous function $f$, at least for large enough $j$.  We discuss this in Section \ref{disc}. 

\begin{condition} \label{selfsim}
Assume that $f: \mathbf M \to [0,\infty)$ is bounded and let $t, b_2>0$ be real numbers. Suppose that there exists a sequence $b(n)$ such that $0 < b(n) \le b_2$ for every $n \in \mathbb N$ and such that $f$ satisfies, for every $j \in \mathcal J_n$, the inequalities 
\begin{equation}\label{condbias}b(n) 2^{-jt} \le \|A_j(f)-f\|_\Omega \le b_2 2^{-jt}.\end{equation}
\end{condition}

Under this condition we can prove asymptotic coverage of our nonparametric confidence band. We should note that inspection of the proof reveals that this coverage result is 'honest': it holds uniformly over classes of densities satisfying Condition \ref{selfsim}.

\begin{theo} [Asymptotic Coverage] \label{covasy}
Let $\Omega$ be any subset of $\mathbf M$. Suppose $f$ satisfies Condition \ref{selfsim} and that the undersmoothing sequence $u_n \in \mathbb N$ is such that $u_n + \frac{1}{t}\log_2(b(n)) \to \infty$ as $n \to \infty$. Then we have, for every $x>0$,
\begin{equation}\liminf_n \Pr \left\{f(y) \in C_n(x,y) ~\textrm{ for every }~y \in \Omega \right\} \ge 1-e^{-x}.
\end{equation} 
\end{theo}

For instance if one knows that $\liminf_n b(n) >0$ (we shall see generic examples for this below) then \textit{any} undersmoothing sequence $u_n \to \infty$ gives asymptotic coverage of the band. On the other hand if $u_n \to \infty$ then $b_n \to 0$ is admissible and the lower bound requirement in Condition \ref{selfsim} becomes more and more lenient as sample size increases. This result and the discussion in Subsection \ref{disc} below shows that our nonparametric procedure does well asymptotically for 'typical' H\"{o}lder-continuous functions on the unit sphere.

\subsubsection{A Nonasymptotic Coverage Result}

The asymptotic Theorem \ref{covasy} is in fact a consequence of the following finite-sample result. While the stochastic terms are similarly well-behaved as in Theorems \ref{uniform} and \ref{poly}, the presence of nonnegligible approximation error is the reason why the following theorem is more intricate.

\begin{theo} [Finite Sample Coverage] \label{cov} Let $\Omega$ be any subset of $\mathbf M$. Suppose $f$ satisfies Condition \ref{selfsim} and let $m^*:=m_n^*(f)$ be the smallest integer such that $b(n)2^{tm^*} \ge 7b_2.$ Set $m:=m_n(f)=\min(j_n^*(t), m^*)$. Then we have, for every $n \in \mathbb N$ and every $x>0$
\begin{equation}\Pr \left\{f(y) \in C_n(x,y) ~\textrm{ for every }~y \in \Omega \right\} \ge 1-e^{-x}-v_n
\end{equation}
where  $$v_n =2(j_{\max}-m)n_2^{-\kappa }+\mathcal I_n $$
with $$\mathcal I_n=I\left\{100 \sqrt{\frac{n_1}{n_2}}\frac{A(n_2, j_n^*(t)+1, \kappa \log n_2)}{A(n_1,j_n^*(t)+u_n-m,x)}  > 2^{(u_n-m-1)(\frac{d}{2}+t)}\right\},$$
with $\kappa>0$ equal to the constant from after (\ref{lepski}) and where $A(n,l,x)$ was defined in (\ref{AAA}).
\end{theo}

\begin{remark} \label{eins} [Undersmoothing in finite samples] \normalfont Note first that if $u_n \ge m$, then the fraction on the l.h.s.~of the inequality in the definition of $\mathcal I_n$ is bounded away from zero and infinity. Consequently the tradeoff between the constants $u_n$ and $b(n)$ is such that if $u_n + t^{-1}\log_2(b(n)) \to \infty$ then $\mathcal I_n=0$ for all $n$ from some $n_0$ onwards, which in particular implies Theorem \ref{covasy}. Not surprisingly obtaining coverage in finite samples is more delicate, as $n_0$ depends on $f$: The undersmoothing constant $u_n$ should be chosen so large that $\mathcal I_n=0$ for every $n$. Closer inspection of $\mathcal I_n$ shows that this is possible if an upper bound for $m$ is available, which can be obtained by requiring an apriori lower bound for the sequence $b(n)$ as well as for $t$. The discussion in Section \ref{disc} will show that such apriori bounds can indeed be obtained in relevant cases.
\end{remark}

\begin{remark} \textit{[Admissible lower bounds in Condition \ref{selfsim}] }\normalfont
Another point of view is to start with an undersmoothing sequence $u_n$ and to ask which sequences of $b(n)$'s are admissible to obtain coverage. Assume for simplicity that the sample size is $2n$ and that $n_1=n_2=n$. Let $C_n(\kappa \log n,y)$ be the confidence band from (\ref{band}) with undersmoothing sequence $u_n \in \mathbb N$ and $x=\kappa \log n$. If $f$ satisfies Condition \ref{selfsim} and if $$b(n) \ge 7b_2 \cdot (100)^{t/(t+d/2)} 2^{(-u_n+2)t},$$ then
\begin{equation}\Pr \left\{f(y) \in C_n(\kappa \log n, y) ~\textrm{ for every }~y \in \Omega \right\} \ge 1-(2j_{\max}+3)n^{-\kappa}.
\end{equation}
For instance if $d=2$ and $f$ is at least once differentiable, then finite sample coverage holds for the set of densities that satisfy Condition \ref{selfsim} for $1 \le t < \infty$ and $b(n) \ge b_2 \cdot 2^{8.2-u_n}$. 
\end{remark}

\begin{remark} \label{kap} \textit{[The role of the thresholding constant $\kappa$] }\normalfont
The thresholding constant $\kappa$  plays an important role in the construction of $\hat j_n$. Our results are presented for fixed $\kappa$ without any restriction on this constant. This is an advantage since this constant has to be carefully chosen in applications. Our bounds typically contain a term of the form $n^{-\kappa}$, and one could be tempted to choose $\kappa$ as large as possible, however it is important to notice that choosing $\kappa$ very large will increase the difficulty of cancelling $\mathcal I_n$ in Theorem \ref{cov}. An adaptive choice of this tuning constant is possible but beyond the scope of this paper.
\end{remark}

\subsection{Regularity of Functions on the Sphere and Condition \ref{selfsim}} \label{disc}

Condition \ref{selfsim} can be characterized in terms of classical H\"{o}lder regularity properties of the unknown density $f:\mathbf M \to \mathbb R$. We shall only discuss the case $\mathbf M = \mathbb S^d$, which is the case of primary statistical interest, but all findings below generalize to $\mathbf M$ with suitable modifications.

There are several ways to approximate unknown functions defined on $\mathbb S^d$, but it is a fortiori not clear whether a given method retrieves the natural intuition that the degree of smoothness of a function $f$ is the driving quantity of the approximation properties of $f$. For instance, while $L^2(\mathbf M)$-projections onto spherical harmonics constitute a way of approximating a continuous function $f:\mathbb S^d \to \mathbb R$, it is well known already from the special case $d=1$ that this approximation may diverge at any given point $x$, which is particularly worrying when one is interested in the local or even uniform behavior of the approximation errors. Furthermore the important question arises whether the approximation method allows for very smooth (for instance infinitely differentiable) functions to be approximated in an optimal way. 

The fact that needlets form a tight frame of $L^2(\bS^d)$ implies good approximation properties in that space, similar to those of the spherical harmonics. Moreover, these approximations are also optimal approximands for differentiable and H\"{o}lder-continuous functions in the uniform norm on $\bS^d$ (as follows from the results in \cite{GP10}), so the upper bound in Condition \ref{selfsim} has a natural interpretation in terms of H\"{o}lder-Zygmund-norms on $\bS^d$. 

The lower bound in Condition \ref{selfsim} is more intricate. The results in \cite{jaffard} and \cite{GineNickl10a} for functions on $\mathbb R$ suggest that this condition should be satisfied if $f$ 'attains $t$ as its H\"{o}lder exponent' viewed as a function on the unit sphere (in fact a slightly stronger requirement is necessary). In the simplest case, if a real-valued function $f$ defined on $\mathbb R$ scales like $|x-x_0|^t$ at some point $x_0$ (if $t>1$ a similar property has to hold for the highest existing derivative), then $f$ attains the H\"{o}lder exponent $t$, and the results in \cite{jaffard} imply that 'quasi every' function (in a Baire sense) in $\mathcal C^t(\mathbb R)$ does this. Indeed Proposition 4 in \cite{GineNickl10a} implies that quasi-every function in $\mathcal C^t(\mathbb R)$ satisfies the lower bound in the $\mathbb R$-analogue of Condition \ref{selfsim} (where $A_j(f)$ has to be replaced by a corresponding wavelet projection). Proving such general results in the case where $f$ is defined on the sphere is technical, mostly since needlets only form a tight frame but not an orthonormal basis. We therefore return to the intuition of H\"{o}lder exponents and show that 'typical' $\alpha$-H\"{o}lder functions on $\bS^d$ satisfy Condition \ref{selfsim}: let us consider spherical analogues of functions on $\mathbb R$ that scale like $|x-x_0|$: If $x_0$ is any point in $\mathbb S^d$, then the zonal functions $d_{\mathbb S^d}(x,x_0)$ or $\left(1-\langle x,x_0 \rangle_{d+1}\right)^{1/2}$ are natural candidates for the class $\mathcal C^1(\mathbb S^d)$. More generally 
$$f_\alpha(x)=\left(1-\langle x,x_0 \rangle_{d+1}\right)^{\alpha/2}$$
for $0<\alpha<\infty, \alpha/2 \notin \mathbb N,$ is a natural candidate for $\mathcal C^\alpha(\bS^d)$. We prove in Proposition \ref{lowbd} below $$b_1 2^{-j\alpha} \le \|A_j(f_\alpha)-f_\alpha\|_\infty \le b_2 2^{-j\alpha}$$ for some fixed constants $0<b_1<b_2<\infty$. Note that obviously, for $\alpha = 2k, \; k\in \bN,  f_\alpha (x) = (1- \langle x, x_0\rangle_{d+1})^k=1-  \cos(d_{\bS^{d}} (x,x_0))^k$ is actually a polynomial on $\bS^d$. 

\section{Proofs for Section \ref{linsup}}

\subsection{Proof of Theorem \ref{uniform}}

If $f$ is the volume element of $\mathbf M$, then  
\begin{equation}\label{bias0}
\|A_j(f)-f\|_\infty =0
\end{equation}
for every $j \ge 0$. Clearly by definition of $\hat j_n$
\begin{equation*}
\Pr \left\{\hat j_n \not=0 \right\} \leq \sum_{l \in \mathcal{J}: l > 0} {\Pr} \left \{ \left\|f_{n_2}(0) - f_{n_2}(l) \right\|_\infty > 4\sigma(n,l) \right \}.
\end{equation*}
Now since $Ef_n(l)=A_l(f)=f$ for every $l \ge 0$, the $l$-th probability is bounded by
\begin{eqnarray*}  
&& {\Pr} \left\{\left\|f_{n_2}(0) -  f_{n_2}(l) -Ef_{n_2}(0) + Ef_{n_2}(l) \right\|_\infty > 4\sigma(n,l) \right\} \notag \\
&& \le \Pr \left \{\left\|f_{n_2}(0) -Ef_{n_2}(0) \right\|_\infty > 2 \sigma(n,l)\right\} \notag + \Pr \left \{\left\|f_{n_2}(l) -Ef_{n_2}(l) \right\|_\infty > 2\sigma(n,l)  \right\} \le 2n_2^{-2\kappa }
\end{eqnarray*}
in view of Proposition \ref{pisneed}, so that $\Pr \{\hat j_n \not=0 \} \leq 2j_{\max} n_2^{-2\kappa }$ follows. To prove the second claim of the theorem, we have from independence of $\hat j_n$ and $f_{n_1}$, from (\ref{bias0}) and from Proposition \ref{pisneed}
\begin{eqnarray*}
&& \Pr \left\{f(y) \in C_n(x,y) ~\textrm{ for every }~y \in \mathbf M \right\} \\
&& = \Pr \left\{\sup_{y \in \mathbf M}\left|\hat f_n (y)-f(y)\right| \le s_n(x) \right \} \\
&& \ge 1 - \Pr \left\{\sup_{y \in \mathbf M}\left|f_{n_1}(\hat j_n+u_n,y)-E_1f_n(\hat j_n+u_n,y)\right| >  \sigma(n_1,\hat j_n+u_n,x) \right\} \\
&& = 1- \sum_{0\le l \le j_{\max} } \Pr \left \{\|f_{n_1}(l+u_n, \cdot)-E_1f_{n_1}(l+u_n, \cdot)\|_\infty >  \sigma(n_1,l+u_n,x) \right\} \Pr \{\hat j_n =l \} \\
&& \ge 1- e^{-x} \sum_{0 \le l \le j_{\max} } \Pr \{\hat j_n =l\} = 1- e^{-x}.
\end{eqnarray*}
The last claim of Theorem \ref{uniform} follows from the first and definition of $\sigma(n,l,x)$.

\subsection{Proof of Theorem \ref{poly}}

Since $j_{\max} \to \infty$ as $n \to \infty$ and since this theorem is of an asymptotic nature we assume $J \le j_{\max}$ in what follows. We recall from (\ref{polyid}) that $f \in E_{2^{2J-1}}$ implies  
\begin{equation}\label{b0}
\|A_l(f)-f\|_\infty =0
\end{equation}
for every $l \ge J$. Then
\begin{equation*}
\Pr \left\{\hat j_n > J \right\} \leq \sum_{l \in \mathcal{J}: l > J} {\Pr} \left \{ \left\|f_{n_2}(J) - f_{n_2}(l) \right\|_\infty > 4\sigma(n,l) \right \},
\end{equation*}
and the $l$'th summand is bounded by 
\begin{eqnarray*}  
&& {\Pr} \left\{\left\|f_{n_2}(J) -  f_{n_2}(l) -Ef_{n_2}(J) + Ef_{n_2}(l) \right\|_\infty > 4\sigma(n,l) \right\} \notag \\
&& \le \Pr \left \{\left\|f_{n_2}(J) -Ef_{n_2}(J) \right\|_\infty > 2 \sigma(n,l)\right\} \notag + \Pr \left \{\left\|f_{n_2}(l) -Ef_{n_2}(l) \right\|_\infty > 2\sigma(n,l)  \right\} \le 2n_2^{-2\kappa }
\end{eqnarray*}
in view of (\ref{b0}) and Proposition \ref{pisneed}. 

For integer $l < J-1$ (so that $2^l<2^{J-1}$) we have 
$$\|A_l(f)-f\|_\infty \ge \inf_{p \in E_{2^{J-2}}}\|p-f\|_\infty \equiv b_1>0$$
since $A_l(f) \in E_{2^{2J-2}}$ and since $E_{2^{2J-2}}$ is a closed proper subspace of $E_{2^{2J-1}}$. By definition we have
\begin{equation} \label{lag}
\Pr (\hat j_n = l) \leq  \Pr \left( \left\|f_{n_2}(l) - f_{n_2}(J) \right\|_\infty \le  4\sigma(n, J) \right).
\end{equation}
The triangle inequality and (\ref{b0}) now give $$\left\|f_{n_2}(l) - f_{n_2}(J) \right\|_\infty \ge \|A_l(f)-f\|_\infty - \left\|f_{n_2}(l) - Ef_{n_2}(l)-f_{n_2}(J)+Ef_{n_2}(J) \right\|_\infty$$
so that the probability in (\ref{lag}) is bounded by
\begin{eqnarray*}
&& \Pr \left( \left\|f_{n_2}(l) - Ef_{n_2}(l)-f_{n_2}(J)+Ef_{n_2}(J) \right\|_\infty \ge b_1 - 4\sigma(n,J) \right) \le \\
&&  \Pr \left(\left\|f_{n_2}(l) - Ef_{n_2}(l) \right\|_\infty \ge \frac{b_1}{2} -2\sigma(n, J) \right) + \Pr \left(\left\|f_{n_2}(J) -Ef_{n_2}(J) \right\|_\infty \ge \frac{b_1}{2} - 2\sigma(n, J) \right).
\end{eqnarray*}
For $n$ large enough depending on $b_1$ we have $2\sigma(n,J) \le b_1/4$ so that Proposition \ref{pisneed} implies, for $J$ fixed, $\Pr\{\hat j_n < J-1\} \le \sum_{0 \le l <J-1} \Pr\{\hat j_n = l\} \le 2Je^{-cn}$ for some constant $c>0$ depending on $b_1, J$ and those constants appearing in the definition of $\sigma(n,l,x)$ that do not depend on $n,l$. Summarizing we deduce $\Pr \{\hat j_n \notin [J-1,J] \} \leq 2j_{\max} n_2^{-2\kappa } + 2Je^{-cn}$ for $n$ large enough. To prove coverage we proceed as in Theorem \ref{uniform}, noting $u_n>1$,
\begin{eqnarray*}
&& \Pr \left\{f(y) \in C_n(x,y) ~\textrm{ for every }~y \in \mathbf M \right\} \\
&& \ge 1 - \Pr \left\{\sup_{y \in \mathbf M}\left|f_{n_1}(\hat j_n+u_n,y)-f(y)\right| >  \sigma(n_1,\hat j_n+u_n,x) \right\} \\
&& \ge 1- 2J e^{-cn} - \\
&& ~~ \sum_{J-1 \le l \le j_{\max} } \Pr \left \{\|f_{n_1}(l+u_n, \cdot)-E_1f_{n_1}(l+u_n, \cdot)\|_\infty >  \sigma(n_1,l+u_n,x) \right\} \Pr \{\hat j_n =l \} \\
&& \ge 1 - 2Je^{-cn} - e^{-x} \sum_{J-1 \le l \le j_{\max} } \Pr \{\hat j_n =l\}  \ge 1- e^{-x} - 2J e^{-cn}
\end{eqnarray*}
where we used (\ref{b0}) and Proposition \ref{pisneed}. The last claim of the theorem is proved as in Theorem \ref{uniform}.

\subsection{Proof of Theorems \ref{covasy} and \ref{cov}} \label{prf1}

We first prove Theorem \ref{cov}. For $f$ satisfying Condition \ref{selfsim} there exists a unique $t:=t(f)$ such that $f$ satisfies Condition \ref{selfsim} for this $t$. Define 
\begin{equation} \label{star}
B(j,t) = b_2 2^{-jt},~~ j_n^*(t)=\min\left\{j\in {\mathcal J \setminus \{0\}}: B(j,t) \le \sigma(n_2,j) \right\}-1.
\end{equation} 
If no $j \in \mathcal J$ exists such that $B(j,t) \le \sigma(n_2,j)$ we set $j_n^*(t)=j_{\max}-1$. We shall assume without loss of generality that $b_2$ is large enough such that $b_2 \ge \sigma (1,0)$. In this way $B(j_n^*(t))\ge \sigma(n_2,j_n^*(t))$ also holds when $j_n^*(t)=0$.

It is easy to see that $j_n^*(t)$ satisfies 
\begin{equation} \label{linrat}
2^{j_n^*(t)} \simeq \left(\frac{n_2}{\log n_2}\right)^{\frac{1}{2t+d}},
\end{equation} 
so is a 'rate optimal' resolution level for estimating $f$ satisfying Condition \ref{selfsim} for the given $t$. The constants in the definition of $j_n^*(t)$ depend only on $b_2, t, a,d, k_2$ and $\|f\|_\infty$. 

\begin{lemma} \label{bandcons} 
a)
For every $n \in \mathbb N$,  
 \begin{equation}\label{optrat1}\Pr (\hat j_n >  j_n^*(t)+1)\le
 2(j_{\max}-j_n^*(t))n_2^{-\kappa} .\end{equation}

b) Let $m:=\min(j_n^*(t), m^*)$ where $m^*$ is the smallest integer such that $(b(n)/b_2) 2^{tm^*} \ge 7.$
Then, for every $j \in \mathcal J$ satisfying $0\le j<j_n^*(t)-m$ and every $n \in \mathbb N$ we have $\Pr (\hat j_n=j) \le 2n_2^{-\kappa }.$ 
As a consequence, for every $n \in \mathbb N$,
\begin{eqnarray} \label{optrat}
  \Pr \left(\hat j_n < j^*_n(t) -m \right)\le 2(j_n^*(t)-m)n_2^{-\kappa }
\end{eqnarray}

\end{lemma}
\begin{proof}
Since this lemma only involves the sample $\cS_2$, we set $n=n_2$ for notational simplicity. We also put $j_n^{*+}=j_n^{*}(t)+1$. If $j_n^{*+}=j_{\max}$ Part a) is proved. Otherwise one has
\begin{equation*}
\Pr (\hat j_n > j_n^{*+}) \leq  \sum_{l \in \mathcal{J}: l > j_n^{*+}} \Pr \left( \left\|f_n(j_n^{*+}) - f_n(l) \right\|_\Omega > 4\sigma(n,l) \right).
\end{equation*}
We first observe that by Condition \ref{selfsim} (noting also $Ef_n(j)=A_j(f)$)
\begin{equation*}  
\left\|f_n(j_n^{*+}) - f_n(l) \right \|_\Omega \leq \left\| f_n(j_n^{*+}) - f_n(l) - Ef_n(j_n^{*+}) + Ef_n(l)  \right\|_\Omega + B(j_n^{*+},t) + B(l,t), 
\end{equation*} 
and that $$B(j_n^{*+},t) + B(l,t) \leq 2 B(j_n^{*+},t) \leq 2\sigma(n,j_n^{*+}) \le 2\sigma(n,l) $$ by definition of $j_n^*(t)$ and since $l > j_n^{*+}$. Consequently, the $l$-th probability in the last sum is bounded by
\begin{eqnarray*}  
&& \Pr \left(\left\|f_n(j_n^{*+}) -  f_n(l) -Ef_n(j_n^{*+}) + Ef_n(l) \right\|_\Omega > 2\sigma(n,l) \right) \notag \\
&& \le \Pr \left(\left\|f_n(j_n^{*+}) -Ef_n(j_n^{*+}) \right\|_\Omega >  \sigma(n,l)\right) \notag  + \Pr \left(\left\|f_n(l) -Ef_n(l) \right\|_\Omega > \sigma(n,l)  \right) \le 2n^{-\kappa }
\end{eqnarray*}
where we have used Proposition \ref{pisneed}. 

To prove the second claim, fix $j< j_n^*(t)-m$. Clearly we only have to consider the case $m=m^*$. Observe that 
\begin{equation} \label{lwbd}
\Pr (\hat j_n = j) \leq  \Pr \left( \left\|f_n(j) - f_n(j^*_n(t)) \right\|_\Omega \le  4\sigma(n, j^*_n(t)) \right).
\end{equation}
Now using Condition \ref{selfsim} and the triangle inequality we deduce $$\left\|f_n(j) - f_n(j^*_n(t)) \right\|_\Omega \ge \frac{b(n)}{b_2} B(j,t) - B(j_n^*(t),t) - \left\|f_n(j) - Ef_n(j)-f_n(j^*_n(t))+Ef_n(j^*_n(t)) \right\|_\Omega$$
so that the probability in (\ref{lwbd}) is bounded by
$$\Pr \left( \left\|f_n(j) - Ef_n(j)-f_n(j^*_n(t))+Ef_n(j^*_n(t)) \right\|_\Omega \ge  \frac{b(n)}{b_2} B(j,t) - B(j_n^*(t),t)- 4\sigma(n,j^*_n(t)) \right).$$
By definition of $j_n^*(t)$ and $B(j,t)$, we have 
\begin{eqnarray*}
\frac{b(n)}{b_2} B(j,t) - B(j_n^*(t),t)&=&\frac{b(n)}{b_2} 2^{t(j_n^*(t)-j)}B(j_n^*(t),t) - B(j_n^*(t),t) \\ 
&>& \left(\frac{b(n)}{b_2} 2^{tm}-1\right)B(j_n^*(t),t) \\
\end{eqnarray*}
as well as $B(j_n^*(t),t) \ge \sigma(n,j_n^*(t))\ge  \sigma(n, j)$ so that the last probability is bounded by 
\begin{eqnarray*}
&& \Pr \left( \left\|f_n(j) - Ef_n(j)-f_n(j_n^*(t))+Ef_n(j_n^*(t)) \right\|_\Omega \ge \left[\left(\frac{b(n)}{b_2} 2^{tm}-1\right)  - 4\right] \sigma(n, j_n^*(t)) \right) \\
&& \le  \Pr \left( \left\|f_n(j) - Ef_n(j) \right\|_\Omega \ge2^{-1} \left(\frac{b(n)}{b_2} 2^{tm}-5 \right) \sigma(n, j) \right) \\
&& ~~ + \Pr \left( \left\|f_n(j_n^*(t)) -Ef_n(j_n^*(t)) \right\|_\Omega \ge 2^{-1} \left(\frac{b(n)}{b_2} 2^{tm}-5\right) \sigma(n, j_n^*(t)) \right)
\end{eqnarray*}
By definition of $m$, the term in brackets is greater than or equal to two, and then -- using Proposition \ref{pisneed} -- the last two probabilities do not exceed $2n^{-\kappa }$. Moreover, 
\begin{eqnarray*} 
\Pr \left(\hat j_n < j^*_n(t) -m \right) &=& \sum_{0 \le j<j_n^*(t)-m} \Pr (\hat j_n = j) \le 2\sum_{0 \le j<j_n^*(t)-m}n^{-\kappa } \notag \\
&\le& 
2(j_n^*(t)-m)n^{-\kappa },
\end{eqnarray*}
which completes the proof.
\end{proof}

\medskip

Combining (\ref{optrat1}) with (\ref{optrat}) we have, for every $n \in \mathbb N$ and for $m$ as in the lemma
\begin{eqnarray} \label{conserv}
 \Pr \{\hat j_n \notin [j_n^*(t)-m, j_n^*(t)+1] \} &\le& 
2[(j_n^*(t)-m)+(j_{\max}-j_n^*(t))]n_2^{-\kappa } \notag \\
&=& 2(j_{\max}-m)n_2^{-\kappa }:= Z_n, 
\end{eqnarray}
a fact we shall use below.

We now prove Theorem \ref{cov}. Denoting by $E_1$ expectation w.r.t.~$\cS_1$, one has by definition of $s_n(x)$ that 
\begin{eqnarray*}
&& \Pr \left\{f(y) \in C_n(x,y) ~\textrm{ for every }~y \in \Omega \right\} \\
&& = \Pr \left\{\sup_{y \in \Omega}\left|\hat f_n (y)-f(y)\right| \le s_n(x) \right \} \\
&& = 1 - \Pr \left\{\sup_{y \in \Omega}\left|\hat f_n(y)-E_1\hat f_n(y) + E_1\hat f_n(y) - f(y)\right| > 1.01\sigma(n_1, \hat j_n+u_n,x) \right \} \\
&& \ge 1 - \Pr \left \{\|\hat f_n-E_1\hat f_n\|_\Omega > \sigma(n_1, \hat j_n+u_n,x) \right\} - \Pr \left \{\|E_1 \hat f_n -f\|_\Omega > 0.01 \sigma(n_1, \hat j_n+u_n,x) \right \} \\
&& = 1- I -II
\end{eqnarray*}
About term $I$: This probability equals, by independence of $f_{n_1}(j,y)$ and $\hat j_n$,
\begin{eqnarray*}
&& \Pr \left \{\|f_{n_1}(\hat j_n+u_n, \cdot)-E_1f_{n_1}(\hat j_n+u_n, \cdot)\|_\Omega >  \sigma(n_1,\hat j_n+u_n,x) \right\} \\
&& = \sum_{0 \le l \le j_{\max}} \Pr \left \{\|f_{n_1}(l+u_n, \cdot)-E_1f_{n_1}(l+u_n, \cdot)\|_\Omega >  \sigma(n_1,l+u_n,x) \right\} \Pr \{\hat j_n =l\} \\
&& \le e^{-x} \sum_{0 \le l \le j_{\max}} \Pr \{\hat j_n =l\} = e^{-x}
\end{eqnarray*}
in view of Proposition \ref{pisneed}. 

About term $II$: Using Condition \ref{selfsim} as well as (\ref{conserv}), and recalling (\ref{AAA}), this quantity equals
\begin{eqnarray*}
&& \Pr \left\{\left\|Ef_{n_1}(\hat j_n+u_n)-f \right\|_\Omega > 0.01\sigma(n, \hat j_n+u_n,x) \right\}  \\
&& \le \Pr\left\{100 b_22^{-t(\hat j_n +u_n)} > \sigma(n_1,\hat j_n+u_n,x)\right\}   \\
&&= \Pr\left\{ 100 \sqrt{n_1}b_2 >  2^{(\hat j_n+u_n)(\frac{d}{2}+t)} A(n_1,\hat j_n+u_n, x) \right\} \\
&& \le \sum_{j_n^*(t)-m \le l \le j_n^*(t)+1} I\left\{ 100 \sqrt{n_1}b_2 >  2^{(l+u_n)(\frac{d}{2}+t)}A(n_1,l+u_n,x)\right\}\Pr\{\hat j_n=l\} + Z_n 
\end{eqnarray*}
\begin{eqnarray*}
&& \le I\left\{ \frac{100 b_2\sqrt{n_1}}{A(n_1,j_n^*(t)+u_n-m,x)} > 2^{(j_n^*(t)+1)(\frac{d}{2}+t)} 2^{(u_n-m-1)(\frac{d}{2}+t)}\right\} + Z_n \\
&& \le I\left\{100 \sqrt{\frac{n_1}{n_2}}\frac{A(n_2,,j_n^*(t)+1,\kappa \log n_2)}{A(n_1,j_n^*(t)+u_n-m,x)}  > 2^{(u_n-m-1)(\frac{d}{2}+t)}\right\} + Z_n\\
\end{eqnarray*}
where we have used that (\ref{star}) implies $$2^{(j^*_n(t)+1)(\frac{d}{2}+t)} \ge \frac{\sqrt n_2 b_2}{A(n_2,j_n^*(t)+1,\kappa \log n_2)}$$ in the last inequality. This proves Theorem \ref{cov}. Theorem \ref{covasy} follows from Theorem \ref{cov} using that tradeoff between $b(n)$ and $u_n$ through the constant $m$ (cf.~also Remark \ref{eins}).

\subsection{Proof of Theorem \ref{adapt}}

The size of the band is $2.02\sigma(n_1,\hat j_n+u_n,x).$ In view of (\ref{optrat1}) -- whose proof only requires the hypotheses of Theorem \ref{adapt} -- we have $\hat j_n \le j_n^*(t)+1$ with probability larger than $1-2 (j_{\max}-j_n^*(t))n_2^{-\kappa }$, so the size of this band is less than or equal to $2.02 \sigma(n_1, j^*_n(t)+u_n+1,x)$ with the same probability bound. The second claim of Theorem \ref{adapt} then follows from definition of $\sigma(n,l,x)$ (cf.~(\ref{AAA})) and of $j_n^*(t)$ (cf.~(\ref{linrat})).

\section{Precise Validity of Condition \protect{\eref{condbias}}}
In this section we investigate  examples of functions verifying condition \eref{condbias} if $\mathbf M = \mathbb S^d$. Let us recall that the projection kernel on $\cH_k(\bS^{d})$ is given by 
$$ L_k(\langle x, y\rangle_{d+1}) = \frac 1{|\bS^{d}|} \left(1+ \frac k\nu\right) C^\nu_k (\langle x, y\rangle_{d+1}), \quad \nu=\frac{d-1}2$$
where
 $ C^\nu_k(x)$ is the corresponding Gegenbauer polynomial. For ease of notation we shall redefine $A_j(x,y) = \sum_{k<2^j} a(k/2^j) L_k(x,y)$, to be in line with the notation in \cite{NPW, pnarco, density}. [For $j \to \infty$ this modification is immaterial.] We
shall use the classical  symbol $$ \forall k \in \bN,  (a)_k = a(a+1)..(a+k-1)  (= \frac{\Gamma(a+k)}{\Gamma(a)}  \hbox{if} ~ -a\not \in \bN),~ (a)_0=1.$$ The following Olindes Rodrigues formula defines the Gegenbauer polynomials and is useful for integration by parts: for $t \in I=[-1,1]$
\begin{equation}\label{OR}
 C_k^\nu(t)=(-1)^k \frac 1{k! 2^k} \frac{ (2\nu)_k}{ (\nu +\frac 12)_k} \frac{ D^k \{(1-t^2)^k) \omega^\nu(t) \}}{\omega^\nu(t)}, \quad \omega^\nu(t) =(1-t^2)^{\nu-1/2}.
\end{equation}
\begin{prop} \label{lowbd}
For $0<\alpha <\infty, \; \frac \alpha2 \not \in \bN,$ we define the following functions:
 $$f_\alpha (x) =\left( \sqrt{1- \langle x, x_0\rangle_{d+1}}\right)^\alpha= \left(\sqrt{1-  \cos(d_{\bS^{d}} (x,x_0))}\right)^\alpha$$ where $d_{\bS^d}$ is the geodesic distance on $\mathbb S^d$. Then there exist constants $ c_1>0, \; c_2>0$ independent of $j$ such that 
$$ c_1 2^{-j\alpha} \leq \| A_j(f_\alpha)-f_\alpha\|_\infty  \leq c_2 2^{-j\alpha}.$$
\end{prop}
\noindent {\bf Proof of the upper bound:}
Let us consider first the case $0 <\alpha \leq 1$.
We have
\begin{align*} |A_j(f_\alpha)(x) - f_\alpha(x)| & = | \int_{\bS^{d-1}}  A_j(x,y)f_\alpha(y) dy - f_\alpha(x)| \\ &= |\int_{\bS^{d}}  A_j(x,y) (f_\alpha(y)  - f_\alpha(x)) dy| \\ &\leq \int_{\bS^{d}}  |A_j(x,y)| |f_\alpha(y)  - f_\alpha(x)| dy
\end{align*}
But 
$$ \forall \theta, \theta' \in [0, \pi], \quad  |\sqrt{1-\cos \theta}- \sqrt{1-\cos \theta'}| = \sqrt 2 | \sin \frac{\theta}2- \sin \frac{\theta'}2|  \leq \frac 1{\sqrt 2} |\theta- \theta'|,$$
so
\begin{align*} |f_1(x)- f_1(y)| &= 
| \sqrt{  1-  \cos(d_{\bS^{d} }(x,x_0)) }- \sqrt{1-  \cos(d_{\bS^{d}} (y,x_0)) }| \\
&\leq \frac1{\sqrt{2}}  |  d_{\bS^{d}} (x,x_0)-d_{\bS^{d}} (y,x_0)| \leq 
\frac1{\sqrt{2}}d_{ \bS^{d} } (x,y)   \end{align*}
And, by the subadditivity of $x\mapsto x^{\alpha}$ for $0<\alpha \leq 1$
$$ |f_\alpha (x)- f_\alpha (y)| = |f_1^\alpha (x)- f_1^\alpha (y)| \leq  |f_1 (x)- f_1 (y)|^\alpha \leq
\frac 1{2^{\alpha/2}}(d_{\bS^{d}} (x,y))^\alpha $$
So, by the integration formula for zonal functions on the sphere (Section 9.1 in \cite{Faraut}):
$$ \forall x \in \bS^{d}, \quad \int_{\bS^{d}}  |A_j(x,y)| |f_\alpha(y)  - f_\alpha(x)| dy \leq 2^{-\alpha/2} \int_{\bS^{d}}  |A_j(\langle x, y\rangle_{d+1}))| (d_{\bS^{d}} (x,y))^\alpha dy $$
$$= 2^{-\alpha/2} |\bS^{d-1}| \int_0^\pi A_j(\cos \theta) \theta^\alpha
(\sin \theta )^{d-1} d\theta \leq 2^{-\alpha/2} |\bS^{d-1}|  \int_0^\pi A_j(\cos \theta) 
 \theta^{d-1+\alpha} d\theta $$
But using the concentration result \cite{pnarco} 
$$ \forall K>0, \exists \;  C_K <\infty , \quad  A_j(\cos \theta ) \leq C_K 2^{jd}[ 1\wedge 1/(2^j\theta)^K]$$ 
Taking $ K>d + \alpha $, we obtain
\begin{align*}\| A_j(f) -f\|_\infty &\leq 2^{-\alpha/2} |\bS^{d-1}|  C_K 2^{jd}\left(\int_0^{2^{-j}}  \theta^{d-1+\alpha} d\theta   + \int_{2^{-j}}^1  \theta^{d-1+\alpha}  \frac 1{(2^j\theta)^K}  d\theta \right) \\
&\leq 2^{-\alpha/2} |\bS^{d-1}| C_K 2^{-j\alpha } \frac{K}{(d+\alpha)(K-d-\alpha)}
\end{align*}

Let us now consider the case $\alpha >1$. Taking $d=2$ the previous proof shows that, on  the classical torus $\bT$
 , for $0 <\alpha \leq 1$
, the  $2\pi-$periodical function 
$ \phi_\alpha( \theta) =  (\sqrt{  1-  \cos \theta }  )^\alpha = 2^\alpha |\sin \frac \theta2|^\alpha $ belongs to $\mathcal C^\alpha(\bT)$. But, if for
$ k$ in $\bN,$ $\alpha$ equals $\alpha=k+ \beta  \leq k+1$, it is clear that $\phi_\alpha(\theta) $ is $k-$times differentiable, and $D^k\phi_\alpha(\theta)
  $ as a linear combination of $C_\infty $ periodical funtions times $|\sin \frac \theta2|^{\beta+j}, \; j=0,1,...,k$ belongs to $\mathcal C^\beta(\bT).$ So, $\phi_\alpha \in \mathcal C^\alpha(\bT),$ and, as  moreover $ \phi_\alpha( \theta)$ is even, there exists $ P_j(\cos \theta)$, a sequence of trigonometrical polynomials of degree
  less than $2^j$ such that :
 $$\|  (\sqrt{  1-  \cos \theta })^\alpha - P_j(\cos \theta)  \|_\infty  \leq C2^{-j\alpha }$$ 
  But $P_j(\cos \langle x, x_0\rangle_{d+1})$ is a polynomial on the sphere of degree less than $2^j$ and $$\|(\sqrt{  1-  \cos \langle x, x_0\rangle_{d+1} })^\alpha -P_j(\langle x, x_0\rangle_{d+1})  \|_\infty  \leq C2^{-j\alpha }.$$ 

\noindent{\bf Proof of the lower bound}
We only have to consider the case $j$ large enough since $f_\alpha$ is not a spherical polynomial and thus not in $E_N(\mathbb S^d)$ for any finite $N$. Using again the integration fomulae for zonal functions
\begin{align*}
 \| A_j(f_\alpha) -f_\alpha\|_\infty &\geq | A_j(f_\alpha)(x_0)  -f_\alpha(x_0)|   \\
 &=    | \int_{\bS^{d}}  A_j(x_0,y) (f_\alpha(y)  - f_\alpha(x_0) ) dy   |
 \\ & =    |\int_{\bS^{d}}  A_j(x_0,y)  (\sqrt{1- \langle y, x_0\rangle_{d+1}}  )^\alpha dy  |
 \\
&= |\bS^{d-1}   | \int_0^\pi A_j(\cos \theta) (\sqrt{1-\cos \theta})^\alpha
(\sin \theta )^{d-1} d\theta | 
\\ & =|\bS^{d-1}  | \int_I   A_j(t) (1-t )^{\alpha/2} (1-t^2)^{\nu-1/2} dt | 
\\
& = \frac{|\bS^{d-1}  |}{  |  \bS^{d}  |}    |  \sum_{0 \leq k  <2^j}  a( \frac{ k}{2^j} ) (1+\frac{ k}{\nu}  )\int_I  C^\nu_k(t) (1-t)^{\alpha/2} (1-t^2)^{\nu-1/2} dt | 
\end{align*}
But, using (\ref{OR})
   \begin{align*}\int_I  C^\nu_k(t) (1-t)^{\alpha/2} (1-t^2)^{\nu-1/2} dt &=\int_I  C^\nu_k(t) (1-t)^{\alpha/2} \omega^\nu(t) dt \\ & =(-1)^k \frac 1{k! 2^k} \frac{ (2\nu)_k}{ (\nu +\frac 12)_k} \int_I (1-t)^{\alpha/2}   D^k \{(1-t^2)^k \omega^\nu(t) \}dt\\
   &=\frac 1{k! 2^k} \frac{ (2\nu)_k}{ (\nu +\frac 12)_k} \int_I  D^k \{(1-t)^{\alpha/2}  \}  (1-t^2)^k \omega^\nu(t) dt  
   \\
   &=\frac 1{k! 2^k} \frac{ (2\nu)_k}{ (\nu +\frac 12)_k} (-\frac \alpha2)_k \int_I  (1-t)^{\alpha/2-k}   (1-t^2)^k \omega^\nu (t) dt  =  \\
   &=\frac 1{k! 2^k} \frac{ (2\nu)_k}{ (\nu +\frac 12)_k} (-\frac \alpha2)_k \int_I  (1-t)^{\alpha/2}   (1+t)^k \omega^\nu (t) dt  = u_k  
   \end{align*}
   Clearly  $ \forall k \geq 0, u_k \neq 0  $ (because $ \frac \alpha2 \not \in \bN) $, $u_k  = (-1)^k |u_k| $ for $ 0 \leq k < \frac \alpha2 + 1$  and $u_k  =- (-1)^{[\alpha/2]}  |u_k|$
for $ k> \frac \alpha2 +1$.  
By the upper bound, and for $j $ large enough : 
$$ C2^{-j} \geq \| A_jf_\alpha-f_\alpha\|_\infty \geq  \frac{|\bS^{d-1}  |}{  |  \bS^{d}  |}   |  \sum_{0\leq k<2^j}  a(\frac k{2^j} ) (1+\frac k\nu) u_k  |  $$
$$=  \frac{|\bS^{d-1}  |}{  |  \bS^{d}  |}   |  \sum_{0\leq k  < \alpha/2 +1}  a(\frac k{2^j} ) (1+\frac k\nu) (-1)^k |u_k| -  (-1)^{[\alpha/2]}    \sum_{ \alpha/2 +1< k<2^j}  a(\frac k{2^j} )(1+\frac k\nu) |u_k| \; | $$
$$=   \frac{|\bS^{d-1}  |}{  |  \bS^{d}  |}  |  \sum_{0\leq k \leq [\alpha/2]}   (1+\frac k\nu) u_k 
 -  (-1)^{[\alpha/2]}    \sum_{ \alpha/2 +1< k<2^j}  a(\frac k{2^j} ) (1+\frac k\nu) |u_k|\; |  $$
So if $[\alpha/2]$ is even, and $j $ large enough 
 $$|  \sum_{0\leq k<2^j}  a(\frac k{2^j} ) (1+\frac k\nu) u_k  | =  
 \sum_{0\leq k \leq [\alpha/2]}   (1+\frac k\nu) u_k
 -      \sum_{ \alpha/2 +1< k<2^j}  a(\frac k{2^j} ) (1+\frac k\nu) |u_k|  $$
 $$=  \sum_{0\leq k \leq [\alpha/2]}   (1+\frac k\nu) u_k +
      \sum_{ \alpha/2 +1< k<2^j}  a(\frac k{2^j} ) (1+\frac k\nu) u_k.  $$
      So
      $$  \sum_{0\leq k<2^{j}}   (1+\frac k\nu)u_k  \leq |  \sum_{0\leq k<2^j}  a(\frac k{2^j} ) (1+\frac k\nu) u_k  | \leq  \sum_{0\leq k<2^{j-1}}   (1+\frac k\nu)u_k.$$
Now, if $[\alpha/2]$ is odd, and $j $ large enough       
 $$|  \sum_{0\leq k<2^j}  a(\frac k{2^j} ) (1+\frac k\nu) u_k  | =  -(
 \sum_{0\leq k \leq [\alpha/2]}   (1+\frac k\nu) u_k
 +     \sum_{ \alpha/2 +1< k<2^j}  a(\frac k{2^j} ) (1+\frac k\nu) |u_k|)  $$
$$= -(
 \sum_{0\leq k \leq [\alpha/2]}   (1+\frac k\nu) u_k
 +     \sum_{ \alpha/2 +1< k<2^j}  a(\frac k{2^j} ) (1+\frac k\nu) u_k)  $$
 So
$$-  \sum_{0\leq k<2^{j}}   (1+\frac k\nu)u_k  \leq |  \sum_{0\leq k<2^j}  a(\frac k{2^j} ) (1+\frac k\nu) u_k  | \leq  -  \sum_{0\leq k<2^{j-1}}   (1+\frac k\nu)u_k,$$
and in any case
 $$|  \sum_{0\leq k<2^j}  a(\frac k{2^j} ) (1+\frac k\nu) u_k  | \sim | \sum_{0\leq k<2^j}   (1+\frac k\nu)\int_I  C^\nu_k(t) (1-t)^{\alpha/2} (1-t^2)^{\nu-1/2} dt | $$
  
Denote now by $\langle \cdot, \cdot \rangle_\nu$ the $L^2([-1,1])$-inner product w.r.t.~$\omega^\nu$  and recall (see \cite{Askey} p.343) $$ \sum_{0 \leq k \leq n}    (1+ \frac k\nu) C_k^\nu (x)  =\frac{  (n + 2\nu) C^\nu_n(x)  - (n+1) C_{n+1}^\nu(x) }{  2\nu (1-x)  }$$
so that
 $$ 2\nu   \langle  \sum_{0 \leq k \leq n}    (1+ \frac k\nu) C_k^\nu (x) ,  (1-x)^{\alpha/2}  \rangle_\nu  = (n + 2\nu)  \langle C^\nu_n (x), (1-x)^{\alpha/2-1} \rangle_\nu  -  (n+1) \langle C_{n+1}^\nu (x), (1-x)^{\alpha/2-1} \rangle_\nu  $$
 $$  \langle C^\nu_k (x), (1-x)^{\alpha/2} \rangle_\nu  = (-1)^k \frac 1{k! 2^k} \frac{ (2\nu)_k}{ (\nu +\frac 12)_k} \int_I (1-t)^{\alpha/2-1} D^k((1-t^2)^k \omega_\nu(t)) dt$$
 $$=   \frac 1{k! 2^k} \frac{ (2\nu)_k}{ (\nu +\frac 12)_k} (1-\frac \alpha2)_k \int_I (1-t)^{\alpha/2-1-k}  (1-t^2)^k (1-t^2)^{\nu-1/2} dt$$
  $$=   \frac 1{k! 2^k} \frac{\Gamma  (2\nu +k)}{\Gamma (2\nu )} \frac{\Gamma(\nu + \frac 12)}{\Gamma(\nu +k + \frac 12)}   \frac{ \Gamma (-\frac \alpha2+k +1)}{\Gamma  (1-\frac \alpha2 )} \int_I (1-t)^{ \nu +\alpha/2 -3/2}   (1+t)^{\nu-1/2 +k} dt$$
  
  $$=   \frac{\sin \frac{ \pi \alpha}2}\pi   \Gamma (\alpha/2) \frac 1{k! 2^k} \frac{\Gamma  (2\nu +k)}{\Gamma (2\nu )} \frac{\Gamma(\nu + \frac 12)}{\Gamma(\nu +k + \frac 12)}  
  \Gamma (-\frac \alpha2+k +1)
    2^{2\nu +k-1 +\frac \alpha2}  \frac{\Gamma(\nu+\frac \alpha2 -\frac 12)  \Gamma (\nu +k +\frac 12)}{\Gamma (2\nu +k +\frac \alpha2)}$$
    
    $$=  \frac{ 2^{\alpha/2} \sin( \frac{ \pi \alpha}2)  \Gamma(\nu+\frac \alpha2 -\frac 12)  \Gamma (\alpha/2)}{\Gamma(\nu) \sqrt \pi  }  
     \frac{   \Gamma ( k+1-\frac \alpha2) \Gamma  (2\nu +k)  }{ k! \Gamma (2\nu +k +\frac \alpha2)}.$$
Using the following standard formulaes 
  $$ \Gamma (1- \alpha/2) \Gamma (\alpha/2) = \frac{\pi}{\sin \pi \alpha/2} \; ; \quad  \Gamma(2\nu) \sqrt \pi = 2^{2\nu-1} \Gamma ( \nu) \Gamma(\nu+ 1/2). $$ We we deduce
    $$  \langle  \sum_{0 \leq k \leq n}    (1+ \frac k\nu) C_k^\nu (x) ,  (1-x)^{\alpha/2}  \rangle_\nu  =$$
  $$=\frac{ 2^{\alpha/2} \sin( \frac{ \pi \alpha}2)  \Gamma(\nu+\frac \alpha2 -\frac 12)  \Gamma (\alpha/2)}{ 2\nu \Gamma(\nu) \sqrt \pi  } 
    \frac 1{n!}.  \frac{  (n+2\nu)   \Gamma ( n+1-\frac \alpha2) \Gamma  (2\nu +n)  }{  \Gamma (2\nu +n +\frac \alpha2) } \{1- \frac{  ( n+1-\frac \alpha2)}{(2\nu +n +\frac \alpha2)} \}$$
     $$=\frac{(2\nu -1+ \alpha) 2^{\alpha/2} \sin( \frac{ \pi \alpha}2)  \Gamma(\nu+\frac \alpha2 -\frac 12)  \Gamma (\alpha/2)}{ 2\nu \Gamma(\nu) \sqrt \pi  }
      \frac{  (n+2\nu)   \Gamma ( n+1-\frac \alpha2) \Gamma  (2\nu +n)  }{(n+ 2\nu +\alpha/2 )n! \Gamma (2\nu +n +\frac \alpha2) }$$
      $$ \sin( \frac{ \pi \alpha}2)  C(\alpha, \nu)\frac{(n+ 2\nu)}{(n+ 2\nu +\alpha/2 )}
       \frac{    \Gamma ( n+1-\frac \alpha2) \Gamma  (2\nu +n)  }{n! \Gamma (2\nu +n +\frac \alpha2) }$$
       Clearly $\sin( \frac{ \pi \alpha}2) $ determines the sign, and by Stirling's formula :
       $$ \frac{    \Gamma ( n+1-\frac \alpha2) \Gamma  (2\nu +n)  }{n! \Gamma (2\nu +n +\frac \alpha2) }  \sim n^{-\alpha}$$
So the lower bound of $ \|A_j(f_\alpha) -f_\alpha \|_\infty$ is of order $2^{-j\alpha}$. \newline

\textit{Acknowledgement.} The second author gratefully acknowledges the hospitality of the Caf\'{e} Br\"{a}unerhof in Vienna.
        
\bibliography{CB}
 \bibliographystyle{abbrvnat}

\bigskip

\noindent \textsc{Universit\'{e} Paris-Diderot} \\
\textsc{Laboratoire de Probabilit\'{e}s et Mod\`{e}les Al\'{e}atoires} \\
\textsc{175, Rue de Chevaleret, 75013 Paris, France} \\
\textsc{Email}: kerk@math.jussieu.fr, picard@math.jussieu.fr \\

\noindent \textsc{University of Cambridge} \\
\textsc{Statistical Laboratory} \\
\textsc{Center for Mathematical Sciences, CB3 0WB, Cambridge, UK} \\
\textsc{Email}: r.nickl@statslab.cam.ac.uk

\end{document}